\documentclass[a4paper,12pt,oneside,reqno]{amsart}
\usepackage[utf8]{inputenc}
\usepackage[a4paper,hmargin={2.4cm,2.4cm},vmargin={3cm,3cm}]{geometry}
\usepackage{amsmath}
\usepackage{amssymb}
\usepackage{amsthm}
\usepackage{graphicx}
\usepackage{hyperref}
\usepackage{tikz}

\renewcommand\MR[1]{\relax\ifhmode\unskip\spacefactor3000
\space\fi \MRhref{#1}{#1}}
\renewcommand{\MRhref}[2]%
{\href{http://www.ams.org/mathscinet-getitem?mr=#1}{MR #2}}
\newtheorem{theorem}{Theorem}[section]
\newtheorem{lemma}[theorem]{Lemma}
\newtheorem{proposition}[theorem]{Proposition}
\newtheorem{corollary}[theorem]{Corollary}
\newtheorem{claim}[theorem]{Claim}

\theoremstyle{definition}

\theoremstyle{remark}

\numberwithin{equation}{section}

\newcounter{constant}
\newcommand{\newconstant}[1]{\refstepcounter{constant}\label{#1}}
\newcommand{\useconstant}[1]{c_{\textnormal{\tiny \ref{#1}}}}
\renewcommand{\chi}{\alpha}

\def\laweq{\overset{\mathrm{law}}=}

\def\Max{{\mathrm{max}}}

\def\RI{\mathrm{RI}}

\def\Var{\mathop{\mathrm{Var}}\nolimits}

\def\diam{\mathop{\mathrm{diam}}\nolimits}
\def\dist{\mathop{\mathrm{dist}}\nolimits}

\def\Cap{\mathop{\mathrm{cap}}\nolimits}

\def\d{\mathrm{d}}

\def\bbone{\boldsymbol 1 }
\def\<{\langle}
\def\>{\rangle}
\def\PZ{P^{\mathbb Z^d}}
\def\EZ{E^{\mathbb Z^d}}
\def\vec#1{{\boldsymbol #1}}
\begin{document}
% TITLE %<<<
\title{Random walks on torus and random interlacements: macroscopic coupling and phase transition}

  \author[J. Černý]{Jiří Černý}
  \address{Jiří Černý,\newline
    Faculty of Mathematics, University of Vienna,\newline
  Oskar-Morgenstern-Platz 1, 1090 Vienna, Austria}
  \email{jiri.cerny@univie.ac.at}

  \author[A. Teixeira]{Augusto Teixeira}
  \address{Augusto Teixeira, \newline Instituto Nacional de Matem\'atica Pura e Aplicada --
  IMPA,\newline Estrada Dona Castorina 110, 22460-320, Rio de Janeiro,
  Brazil}
  \email{augusto@impa.br}

  \date{\today}

\begin{abstract}
  For $d\ge 3$ we construct a new coupling of the trace left by a random
  walk on a large $d$-dimensional discrete torus with the random
  interlacements on $\mathbb{Z}^d$. This coupling has the advantage of
  working up to \emph{macroscopic} subsets of the torus. As an
  application, we show a sharp phase transition for the diameter of the
  component of the vacant set on the torus containing a given point. The threshold where
  this phase transition takes place coincides with the critical value
  $u_\star(d)$ of random interlacements on $\mathbb Z^d$. Our main tool
  is a variant of the \emph{soft-local time} coupling technique of
  \cite{PT12}.
\end{abstract}

\maketitle%>>>
\section{Introduction}%<<<1
\label{s:introduction}

In this paper we study the trace of a simple random walk $X_n$ on a large $d$-dimensional discrete torus $\mathbb T_N^d= (\mathbb Z/N\mathbb Z)^d$ for $d \geq 3$.
In particular, we investigate the percolative properties of its vacant set
\begin{equation}
  \label{e:VNu}
  \mathcal V_N^u
  = \mathbb T_N^d \setminus \{X_0,\dots,X_{\lfloor uN^d \rfloor}\},
\end{equation}
for a fixed $u\in[0,\infty)$ as $N$ tends to infinity.

Intuitively speaking, the parameter $u$ plays the role of a density of
the random walk trace. More precisely, for small values of $u$ and as $N$
grows, the vacant set occupies a large proportion of the torus.
Therefore, $\mathcal{V}^u_N$ should consists of a single large cluster
together with small finite components. In contrast, for large values of
$u$, the asymptotic density of $\mathcal{V}^u_N$ should be small and it
should have been fragmented into small pieces.

In analogy with the Bernoulli percolation behavior, it is actually
expected that there is a phase transition. Namely, there is a critical
value $u_c(d)$ such that the first behavior holds true for all $u<u_c(d)$
and the second for all $u>u_c(d)$, with high probability as $N$ tends to
infinity.

The percolative properties of $\mathcal V^u_N$ have been studied in
several recent works. In \cite{BS08}, the authors showed that, for large
dimensions $d$ and small enough $u > 0$, the vacant set has a (unique, to
  some extent) connected component with a non-negligible density. In
order to understand the vacant set $\mathcal V^u_N$ more in detail,
Sznitman introduced in \cite{Szn10} a model of random interlacements,
which can be viewed as an analogue of the random walk trace in the torus,
but constructed on the infinite lattice $\mathbb{Z}^d$. In
\cite{Szn10,SS09}, it was then shown that the vacant set of random
interlacements exhibit a percolation phase transition at some level
$u_\star(d)\in (0,\infty)$. It is believed that the critical threshold of
the torus, $u_c(d)$ coincides with $u_\star(d)$.

Later, in \cite{Win08}, it was established that as $N$ grows, the set
$\mathcal{V}_N^u$ converges locally in law to the vacant set of random
interlacements $\mathcal{V}^u$, but this didn't have immediate
consequences on the percolative behavior of the $\mathcal V_N^u$. In
\cite{TW11}, a more quantified control of $\mathcal V^u_N$ in terms of
$\mathcal V^u$ improved our understanding of the behavior of the largest
connected component $\mathcal C^\Max_{u,N}$ of $\mathcal V_N^u$. In
particular, it was shown that, for any dimension $d \geq 3$, with high
probability as $N$ goes to infinity:
\begin{itemize}
  \item for $u$ small enough, there is $\varepsilon >0$ such that
  \begin{equation*}
    |\mathcal C_{u,N}^\Max|\ge \varepsilon N^d,
  \end{equation*}
  \item for $u>u_\star(d)$,
  \begin{equation*}
    |\mathcal C_{u,N}^\Max|=o(N^d),
  \end{equation*}
  \item for $u$ large enough, for some $\lambda(u) >0$
  \begin{equation*}
    |\mathcal C_{u,N}^\Max|=O(\log^\lambda N).
  \end{equation*}
\end{itemize}
Note that this implies the existence of a certain transition in the
asymptotic behavior of $\mathcal{V}^u_N$ as $u$ varies. However it was
not known until now where this transition occurs, whether it is sharp, or
whether it is related to the model of random interlacements. The results
of this paper shed more light on this question.

Unfortunately, we are not able to control directly the volume of the
largest connected component $\mathcal C_{u,N}^\Max$. We thus define another
observable that is better suited to our analysis. To this end
we let $P$ to stand for the law of the simple random walk $(X_n)_{n\ge 0}$
on $\mathbb T_N^d$ started from its invariant distribution (which is
  uniform on $\mathbb T_N^d$), and write $\mathcal C_N(u)$ for the
connected component of $\mathcal V_N^u$ containing some given point, say
$0\in \mathbb T_N^d$. We define the observable
\begin{equation}
  \eta_N(u) = P[\diam \mathcal C_N(u) \ge N/4],
\end{equation}
where the diameter is understood in the Euclidean sense, not in the one
induced by the graph $\mathcal C_N(u)$.

Let us point out that the observable $\eta_N(u)$ is \emph{macroscopic},
that is it depends on the properties of the vacant set $\mathcal V_N^u$
in the box of size comparable with $N$.

The next theorem establishes a phase transition for this observable and
gives its asymptotic behavior in terms of related quantities for random
interlacements.

\begin{theorem}
  \label{t:phasetransition}
  The observable $\eta_N(u)$ exhibits a phase transition at $u_\star(d)$.
  More precisely, for $u>u_\star (d)$,
  \begin{equation}
    \label{e:phasetransitionsubcritical}
    \lim_{N\to\infty} \eta_N(u)=0,
  \end{equation}
  and for $u < u_\star(d)$,
  \begin{equation}
    \label{e:phasetransitionsupercritical}
    \lim_{N\to\infty} \eta_N(u)=\eta (u)>0,
  \end{equation}
  where $\eta (u)$ is the probability that $0\in \mathbb Z^d$ is
  contained in the infinite component of the vacant set $\mathcal V^u$ of
  random interlacements at level $u$.
\end{theorem}

The main ingredient of the proof of Theorem~\ref{t:phasetransition} is a
new coupling between $\mathcal V_N^u$ and $\mathcal V_u$ in macroscopic
boxes of the torus which is of independent interest. This is stated
precisely in the following result.

\begin{theorem}
\label{t:toruscouplingweak}
Let $\mathcal B_N = [0, (1-\delta )N]^d$ for some $\delta >0$. Then for
every $u \ge 0$ and $\varepsilon >0$ there exist couplings $\mathbb Q_N$
of the random walk on $\mathbb T_N^d$ with the random interlacements such
that
\begin{equation}
  \lim_{N\to\infty}
  \mathbb Q_N\big[(\mathcal V^{u(1+\varepsilon )}\cap \mathcal B_N)
    \subset (\mathcal V_N^u \cap \mathcal B_N)
    \subset (\mathcal V^{u(1-\varepsilon )}\cap \mathcal B_N)\big] = 1.
\end{equation}
\end{theorem}

We give a more quantitative version of this theorem later (see
  Theorem~\ref{t:toruscoupling}). Observe again that the box
$\mathcal B_N$ is macroscopic, and that $|\mathcal B_N|/N^d$ can be made
arbitrarily close to one. Theorem~\ref{t:toruscouplingweak} thus improves
considerably the best previously known coupling of the same objects
working with boxes of size $N^{1-\varepsilon}$, see \cite{TW11} (cf.~also
  \cite{Bel13} for another related coupling).

The principal tool for the construction of the above coupling is a
streamlined version of the technique of \emph{soft local times}, which
was recently developed in \cite{PT12} in order to prove new decorrelation
inequalities for random interlacements. This technique allows  to couple two
Markov chains so that their ranges almost coincide. Our formulation,
stated as Theorem~\ref{t:couplingmc} below, provides more explicit bounds
on the probability that the coupling fails, and more importantly, it is well
adapted to situations where one can estimate the mixing time of the
chains in question. See introduction to Section~\ref{s:coupling} for more
details.

Let us now briefly describe the organization of this paper. In
Section~\ref{s:notation} we introduce some basic notation and recall
several useful known results. In Section~\ref{s:coupling}, we extend the
soft local times method and prove our main technical result on the
coupling of ranges of Markov chains. The precise version of
Theorem~\ref{t:toruscouplingweak} giving a coupling between the random
walk on $\mathbb{T}^d_N$ and the vacant set of random interlacements is
stated in Theorem~\ref{t:toruscoupling} in Section~\ref{s:toruscoupling}.
Sections \ref{s:rwproperties}--\ref{s:number} provide estimates on the
simple random walk, equilibrium measures, mixing times and the number of
excursions of the walker which are needed in order to apply the results
of Section~\ref{s:coupling}. Finally, Section~\ref{s:proofs} contains the
proofs of our main results. In the appendix we include a suitable version
of classic Chernov bounds on the concentration of additive functionals of
Markov chains.

\section{Notation and some results}%<<<1
\label{s:notation}

Let us first introduce some basic notation to be used in the sequel. We
consider torus $\mathbb T_N^d=(\mathbb Z^d/N\mathbb Z^d)$ which we
identify, for sake of concreteness, with the set
$\{0,\dots,N-1\}^d \subset \mathbb Z^d$. On $\mathbb Z^d$, we
respectively denote by $|\,\cdot\,|$ and $|\,\cdot\,|_\infty$ the
Euclidean and $\ell^\infty$-norms. For any $x\in\mathbb Z^d$ and $r\ge 0$,
we let $B(x,r)=\{y\in \mathbb Z^d:|y-x|\le r\}$ stand for the Euclidean
ball centered at $x$ with radius $r$. Given $K,U\subset \mathbb Z^d$,
$K^c=\mathbb Z^d\setminus K$ stands for the complement of $K$ in
$\mathbb Z^d$ and $\dist(K,U)=\inf\{|x-y|:x\in K, y\in U\}$ for the
Euclidean distance of $K$ and $U$. Finally, we define the inner boundary
of $K$ to be the set $\partial K=\{x\in K:\exists y\in K^c, |y-x|=1\}$,
and the outer boundary of $K$ as $\partial_e K= \partial (K^c)$.
Analogous notation is used on $\mathbb T_N^d$.

We endow $\mathbb Z^d$ and $\mathbb T_N^d$ with the nearest-neighbor
graph structure. We write $P_x$ for the law on
$(\mathbb T_N^d)^{\mathbb N}$ of the canonical simple random walk on
$\mathbb T_N^d$ started $x\in \mathbb T_N^d$, and denote the canonical
coordinate process by $X_n$, $n\ge 0$. We use $P$ to denote the law of
the random walk with a uniformly chosen starting point, that is
$P=\sum_{x\in \mathbb T_N^d} N^{-d}P_x$. We write $P_x^{\mathbb Z^d}$ for
the canonical law of the simple random walk on $\mathbb Z^d$ started from
$x$, and (with slight abuse of notation) $X_n$ for the coordinate
process as well. Finally, $\theta_k$ denotes the canonical shifts of the
walk, defined on either $(\mathbb T_N^d)^{\mathbb N}$ or
$(\mathbb Z^d)^{\mathbb N}$,
\begin{equation}
  \theta_k (x_0,x_1,\dots) = (x_k,x_{k+1},\dots).
\end{equation}

Throughout the text we denote by $c$ positive finite constants whose
value might change during the computations, and which may depend on the
dimension $d$. Starting from Section~\ref{s:rwproperties}, the constants
may additionally depend on $\gamma$, $\chi$ which we will introduce later
(this will be mentioned again when appropriate). Given two sequences
$a_N, b_N$, we write $a_N \asymp b_N$ to mean that
$c^{-1} a_N \leq b_N \leq c a_N$, for some constant $c\ge 1$.

For $K\subset \mathbb Z^d$ finite, as well as for $K\subset \mathbb T_N^d$,
we use $H_K$, $\tilde H_K$ to denote entrance and hitting times of $K$
\begin{equation}
  H_K=\inf\{k\ge 0: X_k\in K\}, \qquad
  \tilde H_K=\inf\{k\ge 1: X_k\in K\}.
\end{equation}
For $K\subset \mathbb Z^d$ we define the equilibrium measure of $K$ by
\begin{equation}
  e_K(x) = P_x^{\mathbb Z^d}[\tilde H_K = \infty] \bbone\{x\in K\}, \qquad
  x\in \mathbb Z^d,
\end{equation}
and the capacity of $K$
\begin{equation}
  \Cap(K) = e_K(K).
\end{equation}
For every finite $K$, $\Cap(K) < \infty$, which allows to introduce the
normalized equilibrium measure
\begin{equation}
  \label{e:bare}
  \bar e_K(\cdot) = (\Cap(K))^{-1} e_K(\cdot).
\end{equation}

Finally, we give an explicit construction of the vacant set of random
interlacements intersected with a finite set $K\subset \mathbb Z^d$. We
build on some auxiliary probability space an i.i.d.~sequence $X^{(i)}$,
$i\ge 1$, of simple random walks on $\mathbb Z^d$ with the initial
distribution $\bar e_K$, and an independent Poisson process
$(J_u)_{u\ge 0}$ with intensity $\Cap(K)$. The vacant set of the random
interlacements (viewed as a process in $u\ge 0$) when intersected with $K$
has the law characterized by
\begin{equation}
  \label{e:ridef}
  (\mathcal V^u\cap K)_{u\ge 0} \laweq \Big(K\setminus \bigcup_{1\le i \le
    J_u} \bigcup_{k\ge 0}  \{X^{(i)}_k\}\Big)_{u\ge 0},
\end{equation}
see, for instance, Proposition 1.3 and below (1.42) in \cite{Szn10}.

\section{Coupling the ranges of Markov chains}%<<<1
\label{s:coupling}

In this section we
construct a coupling of two Markov chains so that their ranges almost
coincide. A method to construct such couplings was recently introduced in
\cite{PT12}, based on the so-called \emph{soft local times}. We will use
the same method to construct the coupling, but propose a new method to
estimate the probability that the coupling fails.

This is necessary since the estimates in \cite{PT12} use considerably the
fact that the Markov chains in consideration have `very strong renewals'.
More precisely the trajectory of the chain can easily be decomposed into
i.i.d.~blocks (of possibly random length). This, together with bounds on
the moment generating function corresponding to one block, allows them to
obtain very good bounds on the error of the coupling, that is on the
probability that the ranges of the Markov chains are considerably different.

In the present paper, we have in mind an application where this `very
strong renewal' structure is not present. We hence need to find new
estimates on the error of the coupling. These techniques combine the
method of soft local times with quantitative Chernov-type estimates on
deviations of additive functionals of Markov chains. An estimate of this
type suitable for our purposes is proved in the appendix.

Similarly as in \cite{PT12}, we will use the regularity of the transition
probabilities of the Markov chain to improve the bounds on the error of
the coupling. In contrast to \cite{PT12} this regularity will be not
expressed via comparing the transition probability with indicator
functions of large balls (see Theorem~4.9 of \cite{PT12}), but by
controlling the variance of the transition probability.

Note also that the estimates on the error of the coupling provided by
Theorems~\ref{t:coupling}, \ref{t:couplingmc} are weaker than the ones
obtained by techniques of \cite{PT12}, when both techniques apply. This
is due to the fact that the Chernov-type estimates mentioned above give
the worst case asymptotic and are not-optimal in many situations.

\medskip

Let us now precise the setting of this section. Let $\Sigma $ be a finite
state space, $P=(p(x,y))_{x,y\in \Sigma }$ a Markov transition matrix,
and $\nu $ a distribution on $\Sigma $. We assume that $P$ is
irreducible, so there exists a unique $P$-invariant distribution $\pi $
on $\Sigma $. The mixing time $T$ corresponding to $P$ is defined by
\begin{equation}
  T=\min\big\{n\ge 0:\max_{x\in \Sigma }\|P^n(x,\cdot)-\pi
  (\cdot)\|_{TV}\big\}\le
  \frac 14.
\end{equation}
where $\|\cdot\|_{TV}$ denotes the total variation distance
$\lVert \nu - \nu' \rVert_{TV} := (1/2) \sum_x |\nu(x) - \nu'(x)|$. We set
\begin{equation}
  \pi_\star = \min_{z\in \Sigma } \pi (z).
\end{equation}

Let $\mu $ be an \textit{a priori} measure on $\Sigma $ with full
support. (This measure is introduced for convenience only, it will
  simplify some formulas later. The estimates that we obtain do not
  depend on the choice of $\mu $.) Let $g:\Sigma \to [0,\infty)$ be the
density of $\pi $ with respect of $\mu $,
\begin{equation}
  \label{e:defg}
  g(x)=\frac {\pi (x)}{\mu (x)}, \qquad x\in \Sigma ,
\end{equation}
and let further $\rho :\Sigma^2 \to [0,\infty)$ be the `transition
density' with respect to $\mu $,
\begin{equation}
  \label{e:transdensity}
  \rho(x,y)=\frac{p(x,y)}{\mu (y)}, \qquad x,y\in \Sigma .
\end{equation}
We use $\rho_y$ to denote the function $x \mapsto \rho (x,y)$ giving the
arrival probability density at $y$ as we vary the starting point. For any
function $f:\Sigma \to \mathbb R$, let
$\pi (f) = \sum_{x\in \Sigma } \pi (x) f(x)$, and
$\Var_\pi f = \pi \big((f - \pi (f))^2\big)$.

The following theorem provides a coupling of a Markov chain with
transition matrix $P$ with an i.i.d.~sequence so that their ranges almost
coincide.

\begin{theorem}
  \label{t:coupling}
  There exists a probability space $(\Omega , \mathcal F, \mathbb Q)$
  where one can construct a Markov chain $(Z_i)_{i\ge 0}$ with transition
  matrix $P$ and initial distribution $\nu $ and an i.i.d.~sequence
  $(U_i)_{i\ge 0}$ with marginal $\pi $ such that for any $\varepsilon$
  satisfying
  \begin{equation}
    \label{e:epsas}
    0<\varepsilon \le  \frac 12 \wedge \min_{z\in \Sigma } \frac{\Var_\pi \rho_z}
    {2\|\rho_z\|_\infty g(z)}
  \end{equation}
  and for any $n \ge 2 k(\varepsilon) T$ we have
  \begin{equation}
    \label{e:coupling}
    \mathbb Q\big[\mathcal G(n ,\varepsilon )^c\big]
    \le
    C\sum_{z\in \Sigma } \Big(e^{-c  n\varepsilon ^2} + e^{-c n
      \varepsilon \frac{\pi (z)}{\nu (z)}}+
    \exp\Big\{-\frac{c   \varepsilon ^2   g(z)^2}{\Var_\pi
      \rho_z }\,\frac{n}{k(\varepsilon )T}\Big\}
    \Big),
  \end{equation}
  where $c,C\in (0,\infty)$ are absolute constants,
  $\mathcal G(n,\varepsilon )$ is the `good' event
  \begin{equation}
    \mathcal G =\mathcal G(n,\varepsilon ) =
    \big\{\{U_i\}_{i=0}^{n(1-\varepsilon)} \subset \{Z_i\}_{i=0}^n \subset
      \{U_i\}_{i=0}^{n(1+\varepsilon) } \big\},
  \end{equation}
  and
  \begin{equation}
    k(\varepsilon)=-\min_{z\in \Sigma } \log_2 \frac{\pi_\star
      \varepsilon^2 g(z)^2}{6 \Var_\pi (\rho_z)}.
  \end{equation}
\end{theorem}

\begin{proof}
  To construct the coupling, we use the same procedure as in \cite{PT12}.
  Let $(\Omega , \mathcal F, \mathbb Q)$ be a probability space on which
  we are given a Poisson point process $\eta = (z_i ,v_i )_{i\ge 1}$ on
  $\Sigma \times [0,\infty)$ with intensity measure $\mu\otimes dx$. On
  this probability space we now construct a Markov chain $(Z_i)_{i\ge 0}$
  and an i.i.d.~sequence $(U_i)_{i\ge 0}$ with the required properties.
  For a more detailed explanation of this construction, see \cite{PT12}.

  Let $G_{-1}(z)=0$, $z\in \Sigma $, and define inductively random
  variables $\xi_k\ge 0$, $Z_k\in \Sigma $, $V_k\ge 0$, and random
  functions $G_k:\Sigma \to [0,\infty)$, $k\ge 0$,
  \begin{align}
    \label{e:ZVa}
    \xi_k &= \inf\{ t\ge 0: \exists (z, v ) \in \eta \setminus
    \{(Z_i,V_i)\}_{i=1}^{k-1} \text{ s.t. } G_{k-1}(z )+ t \rho(Z_{k-1},z )
    \ge v \},\\
    G_k(z) &= G_{k-1}(z) + \xi_k \rho(Z_{k-1},z),\\
    (Z_k,V_k) &= \text{the unique point $(z,v)\in \eta $ such that
    $G_k(z)=v$},
    \label{e:ZV}
  \end{align}
  where we use the convention $\rho (Z_{-1},z)=\nu (z)/\mu (z)$. If the
  point satisfying $G_k(z)=v$ in \eqref{e:ZV} is not unique, we pick one
  arbitrarily. The details of the choice are unimportant, as this occurs
  with zero probability.

  Using a similar construction, on the same probability space, we further
  define random variables $U_k \in \Sigma $, $\tilde \xi_k\ge 0$,
  $W_k\ge 0$ and random functions $\tilde G_k:\Sigma \to [0,\infty)$,
  $k \ge 0$,
  \begin{align}
    \tilde\xi_k &= \inf\{ t\ge 0: \exists (z, v)
      \in \eta \setminus \{(U_i,W_i)\}_{i=1}^{k-1} \text{ s.t. } \tilde
      G_{k-1}(z)+ t g(z) \ge v \},\\ \tilde G_k(z) &= \tilde G_{k-1}(z) +
    \tilde\xi_k g(z),\\ (U_k,W_k) &= \text{the unique point
      $(z,v)\in \eta $ such that $\tilde G_k(z)=v$},
  \end{align}
  where again $\tilde G_{-1} \equiv 0$.

  It follows from \cite[Section 4]{PT12} that $Z=(Z_k)_{k\ge 0}$ is a
  Markov chain with the required distribution, and $U = (U_k)_{k\ge 0}$
  an i.i.d.~sequence with marginal $\pi $. Moreover, the sequences
  $(\xi_k)$ and $(\tilde \xi_k)$ are i.i.d.~with exponential mean-one
  marginal. The sequence $(\xi_k)$ is independent of $(Z_k)$, and
  similarly $(\tilde \xi_k) $ is independent of $(U_k)$.

  We now estimate the probability of $\mathcal G(n,\varepsilon )^c$.
  From the above construction it follows that $\mathbb Q$-a.s.
  \begin{equation}
    \begin{split}
      \label{e:ranges}
      \{Z_i\}_{i=0}^k &= \{z \in \Sigma: \text{ there exists $(z,v)\in
      \eta$ with } G_k(z)\ge v\},\\
      \{U_i\}_{i=0}^k &= \{z \in \Sigma: \text{ there exists $(z,v)\in
      \eta$ with } \tilde G_k(z)\ge v\}.
    \end{split}
  \end{equation}
  Consider the following events
  \begin{equation}
    \begin{split}
      A^- &= \big\{\tilde G_{n(1-\varepsilon )} < (1-\tfrac \varepsilon
      2)ng\big\},\\
      A^+ &= \big\{\tilde G_{n(1+\varepsilon )} > (1+\tfrac \varepsilon
      2)ng\big\},\\
      B &= \big\{n (1-\tfrac \varepsilon 2)g \le G_n \le (1+\tfrac \varepsilon 2)ng\big\}.
    \end{split}
  \end{equation}
  Using \eqref{e:ranges}, it follows that
  $\mathcal G(n,\varepsilon )^c\subset (A^+)^c\cup (A^-)^c\cup B^c$.

  To bound the probability of the events $(A^\pm)^c$ and $B^c$, observe
  first that, by construction, $\tilde G_n = g \sum_{i=1}^n \tilde \xi $.
  As $\tilde \xi_i$'s are i.i.d., the standard application of the
  exponential Chebyshev inequality yields the estimate
  \begin{equation}
    \label{e:poisson}
    \mathbb Q\big[ (A^{\pm})^c\big]\le  e^{-c n\varepsilon ^2}.
  \end{equation}
  To estimate $\mathbb Q[B^c]$, we write $G_n(z)$ as
  \begin{equation}
    \label{e:additive}
    G_n(z)= \xi_0 \frac{\nu (z)}{\mu (z)}+\sum_{i=1}^n \xi_i \rho_z(Z_{i-1})
    =\xi_0 \frac{\nu (z)}{\mu (z)}+\int_0^{\tau_n} \rho_z(\bar
    Z_t) dt,
  \end{equation}
  where $(\bar Z_t)_{t\ge 0}$ is a continuous-time Markov chain following
  the same trajectory as $Z$ with mean-one exponential waiting times, and
  $\tau_n$ is the time of the $n$-th jump of $\bar Z$. It follows that
  $\mathbb Q[B^c]$ can be estimated with help of quantitative estimates
  on the deviations of additive functionals of Markov chains. An estimate
  suitable for our purposes is proved in the appendix.

  To apply this estimate we write
  \begin{equation}
    \begin{split}
      \mathbb Q[B^c] \le \sum_{z\in \Sigma } \bigg\{&
        \mathbb Q\big[\tfrac{\xi_0\nu (z)}{\mu (z)} \ge \tfrac 14\varepsilon n
          g(z)\big] +
        \mathbb Q\big[|\tau_n-n|\ge\tfrac 14 n\varepsilon \big]
        \\&+
        \mathbb Q\Big[\int_0^{n(1+\varepsilon /4)} \rho_{z}(\bar Z_t) dt
          - n\big(1+\tfrac \varepsilon 4\big) g(z)
          \ge \tfrac 14 n\varepsilon g(z)\Big]
        \\&+
        \mathbb Q\Big[\int_0^{n(1-\varepsilon /4)} \rho_{z}(\bar Z_t) dt
          - n\big(1-\tfrac \varepsilon 4\big) g(z)
          \le -\tfrac 14n\varepsilon g(z)\Big]\bigg\}.
    \end{split}
  \end{equation}
  The first term satisfies
  \begin{equation}
    \mathbb Q[\xi_0 \nu (z)/\mu(z)\ge \varepsilon n
    g(z)/4]=e^{-cn\varepsilon \frac{\pi(z)}{\nu(z)}}.
  \end{equation}
  The second term can be bounded using a large deviation argument as in
  \eqref{e:poisson}. The last two terms can be bounded using
  \eqref{e:apph} with $\delta = \varepsilon/(4\pm \varepsilon )$,
  $t=n(1\pm \varepsilon /4)$ and $f= \pm \rho_z$, using also the obvious
  identity $\pi (\rho _z)=g(z)$. The theorem then directly follows, the
  condition \eqref{e:epsas} is a direct consequence of the assumption
  \eqref{e:apphcond} of \eqref{e:apph}.
\end{proof}

The same technique can trivially be adapted to couple the ranges of two
Markov chains: Let $P^1$, $P^2$ be transition matrices of two Markov
chains on a common finite state space $\Sigma $ with respective mixing
times $T^1$, $T^2$, but with the same invariant distribution $\pi $. Let
further $\nu^1$, $\nu^2$ be two initial probability distributions on
$\Sigma$. Similarly as above, we fix an a~priori measure $\mu $, and
define $g(x)=\pi (x)/\mu (x)$, $\rho^i(x,y)=\mu (y)^{-1} p^i(x,y)$, $i=1,2$.

\begin{theorem}
  \label{t:couplingmc}
  There exists a probability space $(\Omega ,\mathcal F, \mathbb Q)$
  where one can define Markov chains $Z^1$, $Z^2$ with respective
  transition matrices $P^1$, $P^2$ and starting distributions $\nu^1$,
  $\nu^2$ such that for every $\varepsilon $ satisfying
  \begin{equation}
    \label{e:epsasmc}
    0<\varepsilon \le  \frac 12 \wedge \min_{i=1,2}\min_{z\in \Sigma } \frac{\Var_\pi \rho^i_z}
    {2\|\rho^i_z\|_\infty g(z)}.
  \end{equation}
  and $n\ge 2k(\varepsilon )(T^1\vee T^2)$ we have
  \begin{equation}
    \label{e:couplingmc}
    \mathbb Q\big[\tilde{\mathcal G}(n ,\varepsilon )^c\big]
    \le
    C\sum_{i=1,2}\sum_{z\in \Sigma } \Big(e^{-c  n\varepsilon ^2}
    +e^{-c n\varepsilon \frac{\pi (z)}{\nu^i(z)}}+
    \exp\Big\{-\frac{c \varepsilon ^2   g(z)^2}{\Var_\pi
      \rho^i_z } \frac{n}{k(\varepsilon )T^i}\Big\}
    \Big),
  \end{equation}
  where $c,C\in (0,\infty)$ are absolute constants,
  $\tilde{\mathcal G}(n,\varepsilon )$ is the event
  \begin{equation}
    \label{e:goodmc}
    \tilde {\mathcal G}(n,\varepsilon ) = \big\{
      \{Z^1_i\}_{i=1}^{n(1-\varepsilon )} \subset
      \{Z^2_i\}_{i=1}^n \subset
      \{Z^1_i\}_{i=1}^{n(1+\varepsilon )}
      \big\},
  \end{equation}
  and
  \begin{equation}
    k(\varepsilon)=-\min_{i=1,2}\min_{z\in \Sigma } \log_2 \frac{\pi_\star
      \varepsilon^2 g(z)^2}{6 \Var_\pi (\rho^i_z)}.
  \end{equation}
\end{theorem}

\section{Coupling the vacant sets}%<<<1
\label{s:toruscoupling}

In this section we state the quantitative version of
Theorem~\ref{t:toruscouplingweak} giving the coupling between the vacant
sets of the random walk and the random interlacements in the macroscopic
subsets of the torus. We then show the connection between
Theorem~\ref{t:couplingmc} and our main result by defining the relevant
finite state space Markov chains.

For technical reasons we should work with `rounded boxes' instead of the
usual ones. Their advantage is that the common
potential-theoretic quantities, like equilibrium measure and hitting
probabilities, are smoother on them; similar smoothing was used in
\cite[Section 7]{PT12}. Let
\begin{equation}
  \label{e:gammachi}
  \gamma\in\Big(\frac 1 {d-1},1\Big)
  \quad\text{and}\quad
  \chi \in \Big(0,\frac 14\Big)
\end{equation}
be two constants that remain fixed through the paper. Set
$L=2N^\gamma +\chi N$, and define the box $B$ with rounded corners
\begin{equation}
  \label{e:BN}
  B=B_N=\bigcup_{x\in[L,N-L]^d\cap \mathbb Z^d} B(x,\chi N).
\end{equation}
Let further $\Delta$ be the set of points at distance at least $N^\gamma$
from $B$,
\begin{equation}
  \label{e:DeltaN}
  \Delta = \Delta_N=\Big(\bigcup_{x\in B_N} B(x,N^\gamma)\Big)^c,
\end{equation}
see Figure~\ref{f:BN} for illustration. We view $B$ and $\Delta$ as
subsets of $\mathbb Z^d$ as well as of $\mathbb T_N^d$ (identified with
  $\{0,\dots,N-1\}^d$).
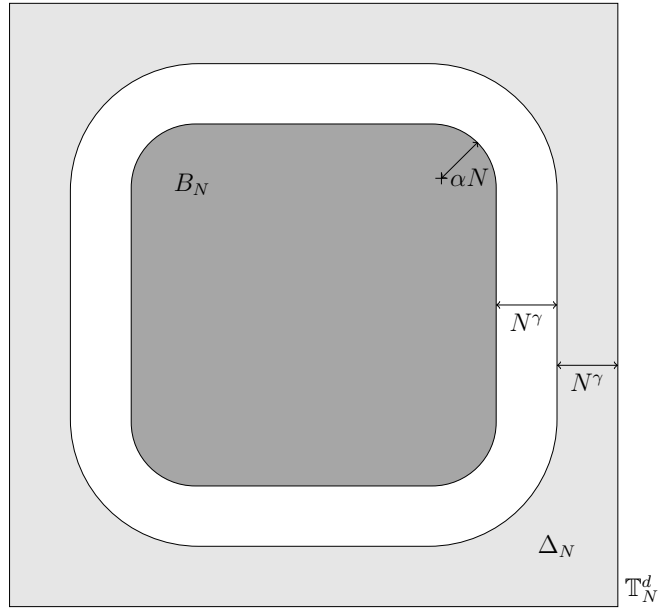
\begin{figure}[ht]
  \centering
  \scalebox{0.8}{
    \begin{tikzpicture}[scale=1]
      \draw[fill=gray!20] (0, 0) rectangle  (10, 10);
      \path[fill=white, rounded corners=60, draw=black] (1,1) rectangle
      (9,9);
      \path[fill=gray!70, draw=black, rounded corners=30] (2,2) rectangle
      (8,8);
      \draw[<->] (7.1,7.1) -- (7.7,7.7);
      \draw (7,7.1) -- (7.2,7.1); \draw (7.1,7) -- (7.1,7.2);
      \node [below right] at (7.1,7.4) {${\chi N}$};
      \draw[<->] (8,5) -- (9,5);
      \node [below] at (8.5,5) {${N^{\gamma}}$};
      \draw[<->] (9,4) -- (10,4);
      \node [below] at (9.5,4) {${N^{\gamma}}$};
      \node at (3,7) {$B_N$};
      \node at (9,1) {$\Delta_N$};
      \node at (10.4,0.3) {$\mathbb T_N^d$};
  \end{tikzpicture}}
  \caption{The rounded box $B_N$ (dark gray), the `security zone' of
    width $N^\gamma $ (white), and the set $\Delta_N$ (light gray) in the
    torus $\mathbb T_N^d$.}
  \label{f:BN}
\end{figure}

We can state the quantitative version of
Theorem~\ref{t:toruscouplingweak} now.
\begin{theorem}
  \label{t:toruscoupling}
  Let $u>0$ and $\varepsilon_N$ be a sequence satisfying
  $\varepsilon_N\in(0,c_0)$ with $c_0$ sufficiently small. Set
  $\kappa = \gamma (d-1)-1>0$ and assume that
  $\varepsilon_N^2 \ge c N^{\delta-\kappa } $ for some $\delta >0$. Then
  there exists coupling $\mathbb Q$ of $\mathcal V_N^u$ with
  $\mathcal V^{u(1\pm \varepsilon_N)}$ such that for every $N$ large enough
  \begin{equation}
    \begin{split}
      \label{e:c1}
      \mathbb Q\big[(\mathcal V^{u(1-\varepsilon_N)} \cap B_N)
        \supset (\mathcal V_N^{u} \cap B_N)
        \supset (\mathcal V^{u(1+\varepsilon_N)} \cap B_N)\big]
      \ge 1- C_1 e^{-C_2 N^{\delta '}}
  \end{split}
  \end{equation}
  for some constants $\delta '>0$, and $C_1,C_2\in (0,\infty)$ depending
  on $u$, $\delta $, $\gamma $ and $\chi $.
\end{theorem}

Theorem~\ref{t:toruscoupling} will be proved with help of
Theorem~\ref{t:couplingmc}. To this end we now introduce relevant Markov
chains which will be coupled together later.

The first Markov chain encodes the excursions of the random walk on the
torus into the rounded box $B$. More precisely, let $R_i$, $D_i$ be the
successive excursion times between $B$ and $\Delta $ of the random walk
$X_n$ on $\mathbb T_N^d$ defined by $D_0=H_\Delta $ and for $i\ge 1$
inductively
\begin{equation}
  \begin{split}
    \label{e:excursions}
    R_i&=H_B\circ \theta_{D_{i-1}} + D_{i-1},\\
    D_i&=H_\Delta \circ \theta_{R_i} + R_i.
  \end{split}
\end{equation}
We define the process
$Y_i=(X_{R_i},X_{D_i})\in \partial B \times \partial\Delta =:\Sigma $,
$i\ge 1$. By the strong Markov property of $X$, $(Y_i)_{i\ge 1}$ is a
Markov chain on $\Sigma $ with transition probabilities
\begin{equation}
  \label{e:Y_i}
  P[Y_{n+1}= \vec y | Y_n=\vec x]
  = P_{x_2}[X_{H_B }=y_1]P_{y_1}[X_{H_\Delta } = y_2],
\end{equation}
for every $\vec x=(x_1,x_2)$ and $\vec y=(y_1,y_2)\in \Sigma$, and with
initial distribution
\begin{equation}
  \label{e:nuY}
  \nu_Y(\vec x)=P[X_{R_1}=x_1, X_{D_1}=x_2] =
  P[X_{R_1}=x_1]P_{x_1}[X_{H_\Delta }=x_2].
\end{equation}

The second Markov chain, encoding the behavior of the random
interlacements in $B$, is defined similarly by considering separately the
excursions of every random walk trajectory of random interlacements which
enters $B$, cf.~\eqref{e:ridef}. Let $(X^{(i)})_{i\ge 1}$ be a
$\PZ_{\bar e_B}$-distributed i.i.d.~sequence, where $\bar e_B$ is the
normalized equilibrium measure of $B$ introduced in \eqref{e:bare}. For
every $i\ge 1$, set $R_1^{(i)}=0$ and define $D_j^{(i)}$, $R_j^{(i)}$,
$j\ge 1$ analogously to \eqref{e:excursions} to be the successive
departure and return times between $B$ and $\Delta $ of the random walk
$X^{(i)}$. Set
\begin{equation}
  \label{e:Ti}
  T^{(i)}=\sup\{j:R^{(i)}_j<\infty\}
\end{equation}
to be the number of excursions of $X^{(i)}$ between $B$ and $\Delta $
which is a.s.~finite. Finally, let $(Z_k)_{k\ge 1}$ be the sequence of
the starting and ending points of these excursions,
\begin{equation}
  \label{e:riexc}
  Z_k=(X^{(i)}_{R^{(i)}_j},X^{(i)}_{D^{(i)}_j}) \qquad
  \text{for $i\ge 1$ and $1\le j
    \le T^{(i)}$ given by } k = \sum_{n=1}^{i-1}T^{(n)}+j.
\end{equation}
The strong Markov property for $X^{(i)}$'s and their independence imply
that $Z_k$ is a Markov chain on $\Sigma $ with transition probabilities
\begin{equation}
  \begin{split}
    \label{e:distZ}
    P & [Z_{n+1}= \vec y | Z_n=\vec x]\\
    & = \big(\PZ_{x_2}[H_B<\infty, X_{H_B}=y_1] +
    \PZ_{x_2}[H_B=\infty]\bar e_B(y_1)\big) \PZ_{y_1}[X_{H_\Delta }=y_2]
  \end{split}
\end{equation}
for every $\vec x,\vec y\in \Sigma $, and with initial distribution
\begin{equation}
  \nu_Z(\vec x)=\bar e_B(x_1)\PZ_{z_1}[X_{H_\Delta}=x_2].
\end{equation}

To apply Theorem~\ref{t:couplingmc}, we need to estimate all relevant
quantities for the Markov chains $Y$ and $Z$. This is the content of the
following four sections.

\medskip

\emph{From now on, all constants $c$ appearing in the text will possibly depend
on the dimension~$d$, and the constants $\chi$ and $\gamma$ defined in
\eqref{e:gammachi}.}

\section{Technical estimates}%<<<1
\label{s:rwproperties}
In this section we show several estimates on potential-theoretic
quantities related to rounded boxes. Let $\bar e_B^\Delta $ be the
normalized equilibrium measure on $B$ for the walk killed on~$\Delta $,
\begin{equation}
  \label{e:eDelta}
  \bar e_B^\Delta (x) =
  \frac {\bbone_{x\in \partial B}} { \Cap_\Delta(B) } P_x[\tilde H_B >
  H_\Delta ],
\end{equation}
where
\begin{equation}
  \label{e:capDelta}
  \Cap_\Delta(B) = \sum_{x\in \partial B} P_x[\tilde H_B > H_\Delta]
\end{equation}
is the associated capacity. We first show that $\bar e_B^\Delta $ is
comparable with the uniform distribution on $\partial B$ and give the
order of $\Cap_\Delta(B)$.

\begin{lemma}
  \label{l:baredelta}
  The is $c\in (0,1)$ such that
  \begin{equation}
    c N^{d-1-\gamma }  \le \Cap_\Delta(B) \le c^{-1}N^{d-1-\gamma },
  \end{equation}
  and for every $x\in \partial B$
  \begin{equation}
    c N^{1-d}\le \bar e^\Delta_B(x) \le c^{-1} N^{1-d}.
  \end{equation}
\end{lemma}
\begin{proof}
  In view of \eqref{e:eDelta}, \eqref{e:capDelta}, to prove the lemma it
  is sufficient to show that uniformly in $x\in \partial B$,
  \begin{equation}
    \label{e:pp}
    c N^{-\gamma } \le P_x[\tilde H_B> H_\Delta ] \le c^{-1}N^{-\gamma }.
  \end{equation}

  For the lower bound, let $\mathcal H_x$ be the $(d-1)$-dimensional
  hyperplane `tangent' to $\partial B$ containing $x$, and let
  $\mathcal H'_x$ be the hyperplane parallel to $\mathcal H_x$ tangent to
  $\partial \Delta $ (see Figure~\ref{f:Hx}). Then
  \begin{equation}
    P_x[\tilde H_B> H_\Delta ]
    \ge P_x[\tilde H_{\mathcal H_x}> H_{\mathcal H'_x}]
    \ge c N^{-\gamma }
  \end{equation}
  where the last inequality follows from observing the projection of $X$
  on the direction perpendicular to $\mathcal H_x$ and the usual
  martingale argument.

  \begin{figure}[ht]
    \centering
    \begin{tikzpicture}[scale=.9]
      \draw (2,0) arc (0:90:2);
      \draw (-2,2) -- (0,2);
      \node [above right] at (-2,4) {$\Delta_N$};
      \draw (4,0) arc (0:90:4);
      \draw (-2,4) -- (0,4);
      \node [below right] at (-2,2) {$B_N$};
      \draw [fill] ({sqrt(2)},{sqrt(2)}) circle(.05);
      \node [above right] at ({sqrt(2)},{sqrt(2)}) {$x$};
      \draw ({2*sqrt(2)},0) -- (0,{2*sqrt(2)});
      \node [above right] at ({2*sqrt(2)},0) {$\mathcal{H}_x$};
      \draw ({4*sqrt(2)},0) -- ({sqrt(2)},{3*sqrt(2)});
      \node [above right] at ({4*sqrt(2)},0) {$\mathcal{H}_x'$};
    \end{tikzpicture}
    \caption{The planes $\mathcal{H}_x$ and $\mathcal{H}'_x$ from the
      proof of Lemma~\ref{l:baredelta}.}
    \label{f:Hx}
  \end{figure}
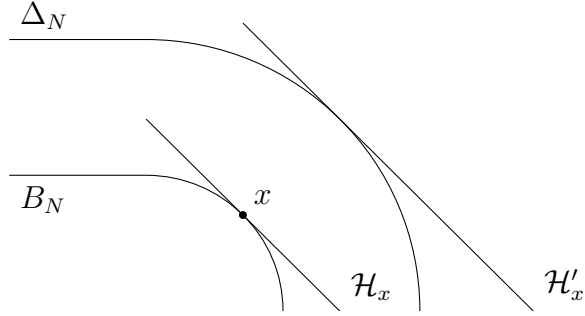

  The upper bound in \eqref{e:pp} is proved similarly. We consider a ball
  $\mathcal G_x$ contained in $B$ with radius $\chi N$ tangent to
  $\partial B$ at $x$, and another ball $\mathcal G'_x$ with radius
  $\chi N+N^\gamma $ concentric with $\mathcal G_x$. Then
  \begin{equation}
    P_x[\tilde H_B> H_\Delta ]
    \le P_x[\tilde H_{\mathcal G_x}> H_{\mathcal G'_x}]
    \le c N^{-\gamma },
  \end{equation}
  using \cite[Proposition~1.5.10]{Law91} and the fact that
  $\chi N\gg N^\gamma $. This completes the proof.
\end{proof}

For the usual equilibrium measure we have similar estimates.
\begin{lemma}
  \label{l:ebar}
  There is a constant $c$ such that for every $x\in \partial B$
  \begin{equation}
    \label{e:ebari}
    c N^{1-d}\le \bar e_B(x) \le c^{-1} N^{1-d}.
  \end{equation}
  and
  \newconstant{c:exit_B}
  \begin{equation}
    \label{e:ebarii}
    \inf_{y\in \partial \Delta }
    P^{\mathbb Z^d}_y[H_B=\infty]\ge \useconstant{c:exit_B} N^{\gamma -1}.
  \end{equation}
\end{lemma}
\begin{proof}
  Since $\Cap(B_N) \asymp N^{d-2}$ (see \cite{Law91}, (2.16) p.53), in
  order to prove the lower bound in \eqref{e:ebari} we need to show that
  $P_x[{\tilde H_B=\infty}]\ge c N^{-1}$. This can be proved by similar
  arguments as above. We fix the hyperplane $\mathcal H_x$ as previously,
  and let $\mathcal H'_x$ be the hyperplane parallel to $\mathcal H_x$ at
  distance $N$. Then
  \begin{equation}
    \label{e:ebarlb}
    P_x[\tilde H_B = \infty] \ge
    P_x[\tilde H_{\mathcal H_x}> H_{\mathcal H'_x}]
    \cdot \inf_{y\in \mathcal H'_x} P_y[H_B=\infty].
  \end{equation}
  By the same reasoning as above, the first term is bounded from below by
  $c N^{-1}$ and the second term is of order constant, as follows easily
  from \cite[Proposition~1.5.10]{Law91} again.

  To prove the upper bound of \eqref{e:ebari}, we need to show that
  $P_x[\tilde H_B=\infty]\le N^{-1}$. To this end fix $\mathcal G_x$ as
  in the previous proof. Then
  \begin{equation}
    P_x[\tilde H_B=\infty]\le P_x[\tilde H_{\mathcal G_x}=\infty]\le
    c N^{-1}
  \end{equation}
  by e.g.~\cite[Lemma~7.5]{PT12}

  Finally, using the same notation as in \eqref{e:ebarlb}, for
  $y \in \partial \Delta$,
  \begin{equation}
    P_y[H_B=\infty]\ge
    P_y[H_{\mathcal H_x}> H_{\mathcal H'_x}]
    \, \inf_{y\in \mathcal H'_x} P_y[H_B=\infty].
  \end{equation}
  The first term is larger than $c N^{\gamma -1}$ by a martingale
  argument and the second is of order constant which proves
  \eqref{e:ebarii} and completes the proof.
\end{proof}

Finally, we control hitting probabilities of boundary points of $B$.
\begin{lemma}
  \label{l:hittingprobas}
  There is a $c<\infty$ such that for every $x\in \partial \Delta $ and
  $y\in \partial B$
  \begin{align}
    \label{e:ubtorus}
    P_x[X_{H_B}=y]&\le c N^{-\gamma (d-1)},\\
    \label{e:ubzd}
    P^{\mathbb Z^d}_x[X_{H_B}=y]&\le c N^{-\gamma (d-1)}.
  \end{align}
  In addition, for every $y\in \partial B$, there are at least
  $c^{-1}N^{\gamma (d-1)}$ points $x\in \partial \Delta $ such that
  \begin{align}
    \label{e:lbtorus}
    P_x[X_{H_B}=y]&\ge c^{-1} N^{-\gamma (d-1)},\\
    \label{e:lbzd}
    P^{\mathbb Z^d}_x[X_{H_B}=y]&\le c^{-1} N^{-\gamma (d-1)}.
  \end{align}
\end{lemma}

\begin{proof}
  The lower bounds \eqref{e:lbtorus}, \eqref{e:lbzd} follow directly from
  \cite[Lemma 7.6(ii)]{PT12} by taking $s=N^\gamma $. The upper bound
  \eqref{e:ubzd} is a consequence of \cite[Lemma 7.6(i)]{PT12}.

  Finally, to show \eqref{e:ubtorus}, let $y_1, y_2\in \partial B$ be two
  points at distance smaller than $\delta N^{\gamma }$ for some
  sufficiently small $\gamma $. By \cite[Proposition 7.7]{PT12}, there is
  a `surface' $\hat D=\hat D(y_1,y_2)$ in $\mathbb Z^d$ separating
  $\{y_1,y_2\}$ from $x$ so that for every $z\in \hat D\setminus B$
  \begin{equation}
    \label{e:comparable}
    c P_z[X_{H_B}=y_1]\le P_z[X_{H_B}=y_2]\le c^{-1}P_z[X_{H_B}=y_1]
  \end{equation}
  for some sufficiently small $c$ independent of $y_1$, $y_2$. Since
  every path in $\mathbb T_N^d\setminus B$ from $x$ to $\{y_1,y_2\}$ must
  pass through $\hat D\setminus B$, using the strong Markov property on
  $H_{\hat D}$, it follows that $z$ can be replaced by $x$ in
  \eqref{e:comparable}. As consequence, for every $y\in \partial B$ there
  are at least $c(\delta N^{\gamma })^{(d-1)}$ points $y'$ on $\partial B$
  with
  \begin{equation}
    P_x[X_{H_B}=y']\ge c P_x[X_{H_B}=y],
  \end{equation}
  from which \eqref{e:ubtorus} easily follows.
\end{proof}

\section{Equilibrium measure}%<<<1
\label{s:equilibrium}
In this section we show that the equilibrium measures of the Markov
chains $Y$ and $Z$ that we defined in Section~\ref{s:toruscoupling}
coincide as required by Theorem~\ref{t:couplingmc}. This may sound
surprising at first, since the periodic boundary conditions in the torus
are felt in the exit probabilities of macroscopic boxes.

\begin{lemma}
  \label{l:invariantmeasure}
  Let $\pi $ be the probability measure on $\Sigma $ given by
  \begin{equation}
    \label{e:pi}
    \pi (\vec x) =   \bar e_B^\Delta (x_1)  P_{x_1}[X_{H_\Delta} =x_2], \qquad
    \vec x=(x_1,x_2)\in \Sigma ,
  \end{equation}
  Then $\pi $ is the invariant measure for both $Y$ and $Z$.
\end{lemma}

\begin{proof}
  To see that $\pi $ is invariant for $Y$ consider the stationary random
  walk $(X_i)_{i\in \mathbb Z}$ (note the doubly infinite time indices)
  on $\mathbb T_N^d$. Let $\mathcal R$ be the set of `returns to $B$' for
  this walk,
  \begin{equation}
    \mathcal R = \{n\in \mathbb Z:X_n\in B, \exists m<n, X_m\in \Delta ,
    \{X_{m+1},\dots,X_{n-1}\}\subset (B\cup \Delta)^c\},
  \end{equation}
  $\mathcal D$ the set of `departures'
  \begin{equation}
    \mathcal D = \{n\in \mathbb Z: X_n\in \Delta , \exists m \in \mathcal
    R, m<n, \{X_m,\dots,X_{n-1}\}\in \Delta^c\},
  \end{equation}
  and write $\mathcal R=\{\bar R_i\}_{i\in \mathbb Z}$,
  $\mathcal D = \{\bar D_i\}_{i\in \mathbb Z}$ so that
  $\bar R_i< \bar D_i < \bar R_{i+1}$, $i\in \mathbb Z$, and
  \begin{equation}
    \label{e:RDbarRD}
    \bar R_0 < \inf\{i\ge 0: X_i\in \Delta \} < \bar R_1.
  \end{equation}
  Observe that by this convention the sequence
  $(\bar R_i,\bar D_i)_{i\ge 1}$ agrees with $(R_i,D_i)_{i\ge 1}$ defined
  in \eqref{e:excursions}. Remark also that $\bar R_0$ might be
  non-negative in general, but $\bar R_{-1}<0$.

  Due to the stationarity and the reversibility of $X$, for every
  $\vec x= (x_1,x_2)$,
  \begin{equation}
    \begin{split}
      \label{e:calR}
      P&[n\in \mathcal R, X_n = x_1]
      \\&= P[X_n=x_1, \exists m<n, X_m\in
        \Delta, \{X_{m+1},\dots,X_{n-1}\}\subset (B\cup \Delta )^c]
      \\& = N^{-d} P_{x_1}[\tilde H_B > H_\Delta ].
    \end{split}
  \end{equation}
  By the ergodic theorem, the stationary measure $\pi_Y$ of $Y$ satisfies
  \begin{equation}
    \pi_Y (\{x_1\}\times \partial \Delta )
    = \lim_{k\to \infty}\frac 1k \sum_{i=1}^k
    \bbone\{X_{R_i}=x_1\}
    = \lim_{m\to\infty}
    \frac{m^{-1}\sum_{n=1}^m \bbone\{n\in \mathcal R, X_n=x_1\}}
    {m^{-1}\sum_{n=1}^m \bbone\{n\in \mathcal R\}},
  \end{equation}
  where we used the observation below \eqref{e:RDbarRD} for the last
  equality. Applying the ergodic theorem for the numerator and denominator
  separately and using \eqref{e:calR} yields
  \begin{equation}
    \pi_Y (\{x_1\}\times \partial \Delta )
    = \frac{P_{x_1}[\tilde H_B>H_\Delta ]}
    {\sum_{y\in \partial B}P_{y}[\tilde H_B>H_\Delta ]}
    = \bar e_B^\Delta (x_1).
  \end{equation}
  By the strong Markov property,
  $\pi_Y(\vec x) = \pi_Y(\{x_1\}\times \partial \Delta ) P_{x_1}[{H_\Delta=x_2}]$
  and thus $\pi_Y=\pi $ as claimed.

  We now consider the Markov chain $Z$. This chain is defined from the
  i.i.d.~sequence of random walks $X^{(i)}$. Each of these random walks
  give rise to a random-length block of excursions distributed as
  $\{(X^{(1)}_{R^{(1)}_i},X^1_{D^{(1)}_i}):i=1,\dots,T^{(1)}\}$. The
  invariant measure $\pi_Z$ of $Z$ can thus be written as
  \begin{equation}
    \pi_Z(\vec x) =
    \frac 1 {\EZ_{\bar e_B} T^{(1)}}
    \EZ_{\bar e_B}\bigg[\sum_{i=1}^{T^{(1)}}
      \bbone_{X^{(1)}_{R^{(1)}_i}=x_1}\bigg]
    {P_{x_1}[X_{H_\Delta }=x_2]}, \qquad \vec x = (x_1,x_2).
  \end{equation}
  To show that $\pi_Z=\pi $ it is thus sufficient to show that the middle
  term is proportional to $P_{x_1}[\tilde H_B > H_\Delta ]$, since the
  first term will then be the correct normalizing factor.

  To simplify the notation we write $X$, $T$, $R_j$ for $X^{(1)}$,
  $T^{(1)}$, $R_j^{(1)}$, and
  extend $X$ to a two-sided random walk on $\mathbb Z^d$ by requiring the
  law of $(X_{-i})_{i\ge 0}$ to be
  $\PZ_{X_0}[\, \cdot \, | \tilde H_B =\infty]$, conditionally
  independent of $(X_i)_{i\ge 0}$. We denote by $L=\sup\{n:X_n\in B\}$
  the time of the last visit of~$X$ to~$B$. Then,
  \begin{equation}
    \begin{split}
      \label{e:Tia}
      &\EZ_{\bar e_B}\Big[\sum_{j=1}^{T}
        \bbone_{X_{R_j}=x_1}\Big]
      =
      \sum_{y\in \partial B}\sum_{z\in \partial B} \bar e_B(y)
      \EZ_y\Big[\bbone_{X_L=z}\sum_{j=1}^{T}
        \bbone_{X_{R_j}=x_1}\Big]
      \\&=
      \sum_{y\in \partial B}\sum_{z\in \partial B} \sum_{n=0}^\infty
      \bar e_B(y)
      \PZ_y\bigg[
        \begin{aligned}
          &X_n=x_1, X_L=z,\\&\exists m\in \mathbb Z: m<n, X_m\in \Delta ,
          \{X_{m+1},\dots,X_{n-1}\}\subset (B\cup \Delta)^c
        \end{aligned}
        \bigg].
    \end{split}
  \end{equation}
  According to \cite[Proposition~1.8]{Szn12}, under $\PZ_{\bar e_B}$,
  $X_L$ has also distribution $\bar e_B$. Hence, by reversibility, this
  equals
  \begin{equation}
    \begin{split}
      &= \sum_{y\in \partial B}\sum_{z\in \partial B} \sum_{n=0}^\infty
      \bar e_B(z)
      \PZ_z\bigg[
        \begin{aligned}
          &X_n=x_1, X_L=y,\\&\exists m>n: X_m\in \Delta ,
          \{X_{n+1},\dots,X_{m-1}\}\subset (B\cup \Delta)^c
        \end{aligned}
        \bigg]
      \\&=
      \sum_{z\in \partial B}
      \sum_{n=0}^\infty
      \bar e_B(z)
      \PZ_z\bigg[
        \begin{aligned}
          &X_n=x_1, \\&\exists m>n: X_m\in \Delta ,
          \{X_{n+1},\dots,X_{m-1}\}\subset (B\cup \Delta)^c
        \end{aligned}
        \bigg]
      \\&=
      \sum_{z\in \partial B}
      \sum_{n=0}^\infty
      \bar e_B(z)
      \PZ_z[X_n=x_1]P_{x_1}[\tilde H_B > H_\Delta ].
    \end{split}
  \end{equation}
  Introducing the Green function $g(x,y)=\sum_{n=0}^\infty \PZ_x[X_n=y]$
  and using the identity $\sum_{z} e_B(z)g(z,x)=1$ (see \cite[Proposition
      1.8]{Szn12}), this equals to
  \begin{equation}
      \label{e:Tib}
      = \sum_{z\in \partial B}
      \bar e_B(z) g(z,x_1)
      P_{x_1}[\tilde H_B > H_\Delta ]
      =
      P_{x_1}[\tilde H_B > H_\Delta ]/\Cap(B).
  \end{equation}
  This shows the required proportionality and completes the proof of the
  lemma.
\end{proof}

We will need the following estimate on the measure $\pi $.

\begin{lemma}
  \label{l:pi_dB}
  For every $y\in \partial \Delta $
  \begin{equation}
    \pi (\partial B \times \{y\}) \le C N^{1-d}.
  \end{equation}
\end{lemma}
\begin{proof}
  By similar arguments as in the proof of Lemma~\ref{l:invariantmeasure},
  using the same notation,
  \begin{equation}
    \begin{split}
      P[n\in \mathcal D, X_n=y]&=P\big[X_n=y,\exists\,m<n,X_m\in B,
        \{X_{m+1},\dots,X_{n-1}\}\in (B \cup \Delta )^c\big]
      \\&=N^{-d} P_y[\tilde H_\Delta > H_B]
      \\&\le c N^{-d-\gamma },
    \end{split}
  \end{equation}
  since, by the same argument as in the proof of Lemma~\ref{l:baredelta},
  $P_y[\tilde H_\Delta > H_B]\le c N^{-\gamma }$. Further,
  \begin{equation}
    P(n\in \mathcal D)=P(n\in \mathcal R)=\sum_{x\in \partial B} N^{-d}
    P_x[\tilde H_B > H_\Delta ] \asymp c N^{-\gamma -1},
  \end{equation}
  by the estimates in the proof of Lemma~\ref{l:baredelta} again.
  Therefore,
  \begin{equation}
    \pi (\partial B \times \{y\})=P[X_n=y|n\in \mathcal D] \le c N^{1-d},
  \end{equation}
  and the proof is completed.
\end{proof}

\section{Mixing times}%<<<1
\label{s:mixing}
The next ingredient of Theorem~\ref{t:couplingmc} are the mixing times
$T_Y$ and $T_Z$ of the Markov chains $Y$ and $Z$. They are estimated in
the following lemma.

\begin{lemma}
  \label{l:mixingtimes}
  There is a constant $c$ such that
  \begin{align}
    T_Z &\le c N^{1-\gamma },\label{e:CPRI}
    \\
    T_Y &\le c N^{1-\gamma}.
    \label{e:CPRW}
  \end{align}
\end{lemma}
\begin{proof}
  To bound the mixing times we use repeatedly the following lemma which
  can be found e.g.~in \cite[Corollary~5.3]{LPW09}.

  \begin{lemma}
    \label{l:couplingmixing}
    Let $(\mathcal X_i)_{i\ge 0}$ be an arbitrary Markov chain on a
    finite state space $\Sigma $. Assume that for every $x,y\in \Sigma $
    there exist a coupling $Q_{x,y}$ of two copies
    $\mathcal X, \mathcal X'$ of $\mathcal X$ starting respectively from
    $x$ and $y$, such that
    \begin{equation}
      \label{e:boundmix}
      \max_{x,y \in \Sigma } Q_{x,y}
      [\mathcal X_n \neq \mathcal X'_n] \leq 1/4.
    \end{equation}
    Then $T_{\mathcal X} \le n$.
\end{lemma}

To show \eqref{e:CPRI}, we thus consider two copies $Z_i$, $Z_i'$ of the
chain $Z$ starting respectively in $\vec x, \vec x'\in \Sigma $ and
define the coupling $Q_{\vec x,\vec x'}$ between them as follows. Let
$(\xi_i)_{i\ge 0}$ be a sequence of i.i.d. Bernoulli random variables
with $P[\xi_i=1]=\useconstant{c:exit_B} N^{\gamma -1}:=p_N$ where the
constant $\useconstant{c:exit_B}$ is as in~\eqref{e:ebarii}.
Given $Z_i=\vec x_i$, $Z'_i=\vec x_i'$, and given $\xi_i=1$ we choose
$Z_{i+1}=Z'_{i+1}$ distributed as
$\nu(\vec x) =\bar e_B(x_1)P_{x_1}[X_{H_\Delta } = x_2]$. On the other
hand, when $\xi_i=0$, we choose $Z_{i+1}$ and $Z'_{i+1}$ independently
with respective distributions $\mu_{\vec x_i} $ and $\mu_{\vec x_i'}$
where (cf.~\eqref{e:distZ})
  \begin{equation}
    \mu_{\vec x}(\vec y) =
    \big\{\PZ_{x_2}[H_B<\infty, X_{H_B}=y_1] +
      (\PZ_{x_2}[H_B=\infty]-p_N)\bar e_B(y_1)\big\}
    \frac{\PZ_{y_1}[X_{H_\Delta }=y_2]}{1-p_N}.
  \end{equation}
  The bound \eqref{e:ebarii} ensures that this is a well-defined
  probability distribution. If $Z_i=Z'_i$ for some $i$, then we let them
  move together, $Z_j=Z'_j$ for all $j\ge i$.

  It follows that
  \begin{equation}
    \max_{\vec x,\vec x'}Q_{\vec x,\vec x'}[Z_i\neq Z_{i'}]\le
    \mathbb P[\xi_j =0 \,\forall j<i] = (1-p_N)^j.
  \end{equation}
  Choosing now $j=c N^{1-\gamma }$ with $c$ sufficiently large and using
  Lemma~\ref{l:couplingmixing} yields \eqref{e:CPRI}.

  To show \eqref{e:CPRW}, let
  $G=G_N=\{x\in B_N: \dist(x,\partial B_N)\ge \chi N/2\}$. Intuitively,
  the excursions of the random walk into $G$ will play the same role as
  the `excursions of the random interlacements to infinity' played in the
  proof of \eqref{e:CPRI}. We need two technical claims
  \begin{claim}
    \label{c:YregenA}
    For some constant $c_1>0$ and all $N$ large,
    \begin{equation}
      \inf_{x\in \partial B}P_x[H_G<H_\Delta ]\ge c_1 N^{\gamma -1}.
    \end{equation}
  \end{claim}
  \begin{proof}
    Similarly as in Section~\ref{s:rwproperties}, let $\mathcal G_x$ be
    the ball with radius $\chi N$ contained in $B$ tangent to $\partial B$
    at $x$, and let $\mathcal G^1_x$, $\mathcal G^2_x$ be the balls
    concentric with $\mathcal G_x$ with radius $\chi N/2$ and
    $\chi N + N^\gamma $ respectively. Then
    $\mathcal G^2_x\subset \mathbb T_N^d\setminus \Delta$, and
    $\mathcal G^1_x\subset G$. Hence, using again
    \cite[Proposition~1.5.10]{Law91},
    \begin{equation}
      P_x[H_G<H_\Delta]\ge P_x[H_{\mathcal G_x^1}<H_{\mathcal G_x^2}]\ge c
      N^{1-\gamma }
    \end{equation}
    which shows the claim.
  \end{proof}
  \begin{claim}
    \label{c:YregenB}
    For some $c_2<\infty$ and all $N$ large,
    \begin{equation}
      \sup_{x\in \partial G}P_x[X_{H_\Delta } = y]
      \le c_2 \inf_{x\in \partial G}P_x[X_{H_\Delta } = y] \qquad
      \text{for all $y\in \partial \Delta $.}
    \end{equation}
  \end{claim}
  \begin{proof}
    For every $y\in \partial \Delta$, the function
    $x\mapsto P_x[X_{H_\Delta}=y]$ is harmonic on
    $\mathbb T_N^d\setminus \Delta$. The claim then follows by Harnack
    principle, see e.g.~\cite[Theorem~1.7.6]{Law91}.
  \end{proof}

  We continue the proof of \eqref{e:CPRW}. For $x\in \partial B$, let
  $\nu_x(\cdot)=P_x[X_{H_{G\cup \Delta}}\in {}\cdot{}]$. By
  Claim~\ref{c:YregenA}, $\nu_x(\partial G)\ge c_1 N^{\gamma -1}$, so we
  can find a sub-probability $\nu^\circ_x$ on $\partial G$ such that
  $\nu^\circ_x(\partial G)=c_1 N^{\gamma -1}$ and $\nu^\circ_x \le \nu_x$.
  For any $x\in \mathbb T_N^d$, let
  $\mu_x(\cdot) = P_x[X_{H_\Delta} \in {}\cdot{}]$, and let $\mu $ be the
  sub-probability on $\partial \Delta $ given by
  $\mu (y)=\inf_{x\in \partial G} \mu_x(y)$. It follows from
  Claim~\ref{c:YregenB} that $\mu (\partial \Delta)\ge c_2^{-1}$. For any
  non-trivial sub-probability measure $\kappa $, we denote by
  $\overline\kappa$ the probability measure obtained by normalizing
  $\kappa $.

  We an now construct the coupling required for the application of
  Lemma~\ref{l:couplingmixing}. Let $\vec x(0),\vec x'(0)\in \Sigma $ and
  define the coupling $Q_{\vec x, \vec x'}$ of two copies $Y$, $Y'$ of $Y$
  as follows. Let $Y_0=\vec x$, $Y'_0=\vec x'$, and let $(\xi_i)_{i\ge 0}$,
  $(\tilde \xi_i)_{i\ge 0}$ be two independent sequences of
  i.i.d.~Bernoulli random variables with $P[\xi_i=1]=c_1 N^{\gamma -1}$
  and $P[\tilde \xi_i=1]=\mu (\partial \Delta)$. Now continue inductively
  through the following steps
  \begin{enumerate}
    \item Given $Y_{k-1}=(Y_{k-1,1},Y_{k-1,2})$ and
    $Y'_{k-1}=(Y'_{k-1,1},Y'_{k-1,2})$, $k\ge 1$, choose $Y_{k,1}$,
    resp.~$Y'_{k,1}$, independently from $P_{Y_{k-1,2}}[X_{H_B}\in {}\cdot{}]$,
    resp.~$P_{Y'_{k-1,2}}[X_{H_B}\in {}\cdot{}]$.

    \item If $\xi_k=0$, choose $U_k$ according to
    $\overline{\nu_{Y_{k,1}} - \nu^\circ_{Y_{k,1}}}$, then $Y_{k,2}$
    according to $\mu_{U_k}$, and analogously $U'_k$ according to
    $\overline{\nu_{Y'_{k,1}} - \nu^\circ_{Y'_{k,1}}}$ and then $Y'_{k,2}$
    according to $\mu_{U'_k}$, independently.

    \item Otherwise, if $\xi_k=1$, choose $U_k$ according to
    $\overline{\nu_{Y_{k,1}}^\circ}$, and $U'_k$ according to
    $\overline{\nu_{Y'_{k,1}}^\circ}$, independently. If, in addition
    $\tilde \xi_k=1$, choose $Y_{k,2}=Y'_{k,2}$ according to
    $\overline \mu$. Otherwise, if $\tilde\xi_k=0$, choose $Y_{k,2}$
    according to $\overline{\mu_{U_k}-\mu }$, and $Y'_{k,2}$ according to
    $\overline{\mu_{U'_k}-\mu }$, independently.

    \item Finally, if for some $k$, $Y_{k,2}=Y'_{k,2}$, let $Y$ and $Y'$
    follow the same trajectory after $k$.
  \end{enumerate}
  It can be checked easily that these steps construct two copies of $Y$
  started from $\vec x$ and $\vec x'$ respectively. Moreover,
  \begin{equation}
    Q_{\vec x,\vec x'}[Y_k \neq Y'_k]
    \le \mathbb P[\xi_i\tilde \xi_i=0 \,\forall i<k]
    = (1-c_1N^{\gamma -1}\mu (\partial \Delta))^{k-1}.
  \end{equation}
  Observing that $\mu (\partial \Delta)\ge c_2^{-1}$, \eqref{e:CPRW}
  follows by taking $k = c N^{1-\gamma }$ with $c$ large enough and using
  Lemma~\ref{l:couplingmixing}.
\end{proof}

\section{Variance estimate}%<<<1
\label{s:variance}
We continue to estimate the ingredients for the application of
Theorem~\ref{t:couplingmc}. Due to the form of the equilibrium measure
$\pi $ introduced in \eqref{e:pi}, it is suitable to fix the base measure
$\mu $ on $\Sigma $ as
\begin{equation}
  \mu (\vec x)=P_{x_1}[X_{H_\Delta }=x_2], \qquad \vec x=(x_1,x_2)\in \Sigma .
\end{equation}
Then (cf.~\eqref{e:defg},\eqref{e:transdensity} for the notation)
\begin{align}
  g(\vec x)&= \bar e_B^\Delta (x_1),
  \\
  \label{e:rhoY}
  \rho^Y(\vec x,\vec y)
  &= P_{x_2}[X_{H_B} =y_1]=:\tilde \rho^Y(x_2,y_1)
  \\
  \label{e:rhoZ}
  \rho^Z(\vec x, \vec y)
  &= P^{\mathbb Z^d}_{x_2}[X_{H_B} =y_1]
  + P^{\mathbb Z^d}_{x_2}[H_B =\infty] \bar e_B(y_1)
  =:\tilde \rho^Z(x_2,y_1).
\end{align}
Recall that $\rho^\circ_\vec x$ denotes the function
$\vec y \mapsto \rho^\circ(\vec y,\vec x)$; we use $\circ$ to stand for
either $Y$ or $Z$.

\begin{lemma}
  \label{l:varrho}
  There exist constants $c,C\in (0,\infty)$ such that and for every
  $\vec x\in \Sigma $
  \begin{equation}
    \label{e:81}
    c N^{1-d} N^{-\gamma (d-1)}
    \le \Var_\pi \rho^\circ_{\vec x}
    \le C N^{1-d} N^{-\gamma (d-1)}.
  \end{equation}
\end{lemma}
\begin{proof}
  An easy computation yields, using Lemma~\ref{l:pi_dB} for the last
  inequality,
  \begin{equation}
    \label{e:gbound}
    \begin{split}
      \Var_\pi \rho^\circ_{\vec x}
      &\le \sum_{\vec x'\in \Sigma } \pi (\vec x')
      \rho^\circ(\vec x',\vec x)^2
      \\&= \sum_{x'_2 \in \partial \Delta } \pi (\partial B\times \{x'_2\})
      \tilde \rho^\circ(x'_2,x_1)^2
      \\&\le C N^{1-d}\sum_{x'_2\in \partial \Delta }  \tilde \rho^\circ(x'_2,x_1)^2
    \end{split}
  \end{equation}
  Using Lemmas~\ref{l:ebar}, \ref{l:hittingprobas} in \eqref{e:rhoY} and
  \eqref{e:rhoZ}, we obtain that
  \begin{equation}
    \label{e:maxrho}
    \max_{x\in \partial B, y\in \partial \Delta }\tilde \rho ^\circ(y,x)\le c N^{-\gamma(d-1)}
  \end{equation}
  for both chains $\circ \in \{Y,Z\}$. Therefore
  \begin{equation}
    \Var_\pi \rho^\circ_{\vec x}
    \le C N^{1-d} \sup\Big\{\sum_{z\in \partial B} h^2(z)\,:\,
    {h\!:\!\partial B \to [0,cN^{-\gamma(d-1)}]},\sum_{z\in \partial B}
    h(z)=1\Big\} .
  \end{equation}
  The supremum is achieved by a function $h$ that takes the maximal value
  $cN^{-\gamma (d-1)}$ for as many points as it can, by a convexity
  argument. Hence,
  \begin{equation}
    \Var_\pi \rho^\circ_{\vec x}\le CN^{1-d} N^{\gamma (d-1)} (N^{-\gamma (d-1)})^2,
  \end{equation}
  and the upper bound follows.

  Finally, by Lemma~\ref{l:hittingprobas} and \eqref{e:rhoY},
  \eqref{e:rhoZ}, for every $x\in \partial B$ there are at least
  $c N^{\gamma (d-1)}$ points ${y\in \partial \Delta} $ such that
  $\tilde \rho^\circ(y,x)\ge c'N^{-\gamma (d-1)}$. Hence,
  $\pi \big((\rho_{\vec x}^\circ)^2\big)$ is larger than the left-hand side of
  \eqref{e:81}. Moreover, since $\pi$ is invariant for both Markov chains,
  it follows that $\pi (\rho_{\vec x}^\circ)^2=g(x)^2 \asymp N^{2(1-d)} $, by
  Lemma~\ref{l:baredelta}. Combining the last two claims, the lower bound
  follows.
\end{proof}

\section{Number of excursions}%<<<1
\label{s:number}
The final ingredient needed for Theorem~\ref{t:couplingmc} is an estimate
on the number of excursion that the random walk typically makes before
the time $uN^d$, as well as on the corresponding quantity for the random
interlacements at level $u$.

Consider first the random walk on the torus. Define
\begin{equation}
  \mathcal N(t)=\sup\{i:R_i < t\}
\end{equation}
to be the number of excursions starting before $t$. We show that
$\mathcal N(t)$ concentrates around its expectation.

\begin{proposition}
  \label{p:concentr_Nt}
  Let $u > 0$ be fixed. There exist constants $c,C$ depending only on
  $\gamma $ and $\chi $ such that for every $N\ge 1$
  \begin{equation}
    P\big[ \big|\mathcal{N}(u N^d) - u \Cap_\Delta(B)\big| > \eta
    \Cap_\Delta (B)\big] \leq C \exp\{-c \eta^2 N^c\}.
  \end{equation}
\end{proposition}

\begin{proof}
  To prove the proposition we first compute the expectation of
  $\mathcal N(t)$.
  \begin{lemma}
    For every $t\in \mathbb N$,
    \begin{equation}
      \label{e:ENt}
      |E \mathcal N(t) - t N^{-d} \Cap_\Delta (B)|\le 1.
    \end{equation}
    Moreover, when starting from $\bar{e}^\Delta_B$, the stationary
    measure for $R_i$'s, we have
    \begin{equation}
      \label{e:ERk}
      E_{\bar{e}^\Delta_B}(R_1) = \frac{N^d}{\Cap_\Delta(B)}.
    \end{equation}
  \end{lemma}

  \begin{proof}
    Recall from the proof of Lemma~\ref{l:invariantmeasure} that
    $(\bar R_i,\bar D_i)$ denote the returns and departures of the
    stationary random walk $(X_n)_{n\in \mathbb Z}$. Let
    $\bar{\mathcal N}(t)= \sup\{i:\bar R_i<t\}$. By the observation below
    \eqref{e:RDbarRD}, $|\bar{\mathcal N}(t)-\mathcal N(t)|\le 1$. It is
    thus sufficient to show that
    $E\bar{\mathcal N}(t)=tN^{-d}\Cap_\Delta (B)$. To this end recall
    equality \eqref{e:calR}. Summing it over $x_1\in \partial B$, we obtain
    \begin{equation}
      P[{k=\bar R_j}\text{ for some }j]=N^{-d}\Cap_\Delta (B),\qquad k\ge 0.
    \end{equation}
    The required claim follows by summation over $0 \le k < t$.

    The second claim of the lemma is a consequence of the first claim,
    the fact that every $X_{R_k}$ is $\bar e_B^\Delta $-distributed at
    stationarity, and the ergodic theorem.
    %To prove \eqref{e:ERk}, we start observing that
    %\begin{equation}
    %  \frac{R_{\mathcal{N}(t)}}{\mathcal{N}(t)} \leq
    %  \frac{t}{\mathcal{N}(t)} \leq \frac{R_{\mathcal{N}(t) +
    %  1}}{\mathcal{N}(t)+1} \frac{\mathcal{N}(t)+1}{\mathcal{N}(t)}.
    %\end{equation}
    %Taking the limits on both sides we obtain
    %\begin{equation}
    %  \frac{\mathcal{N}(t)}{t} \to \frac{1}{E_{\bar{e}^\Delta_B}(R_1)},
    %\end{equation}
    %finishing the proof of the lemma.
  \end{proof}

  We proceed with proving Propositions~\ref{p:concentr_Nt}. It is more
  convenient to show a concentration result for the return times $R_i$
  instead of $\mathcal{N}(t)$. Observing that for any $t > 0$ and $b > 0$,
  \begin{equation}
    \big\{|\mathcal{N}(t) - E(\mathcal{N}(t))| > b \big\} \subseteq
    \big\{ R_{\lceil E(\mathcal{N}(t)) - b \rceil} > t \big\} \cup \big\{
    R_{\lfloor E(\mathcal{N}(t)) + b \rfloor} > t \big\}
  \end{equation}
  we obtain easily that
  \begin{equation}
    \label{e:97}
    P\big[ \big|\mathcal{N}(u N^d) - u \Cap_\Delta (B) \big| > \eta
    \Cap_\Delta (B) \big] \leq P[R_{k_-} > uN^d] + P[R_{k_+} < uN^d],
  \end{equation}
  where $k_- = \lceil (u-\eta) \Cap_\Delta (B)\rceil$ and
  $k_+ = \lfloor(u+\eta) \Cap_\Delta (B)\rfloor$.

  Let $\varepsilon > 0$ be a small constant that will be fixed later, and
  set $\ell = \lfloor N^{\varepsilon} T_Y \rfloor$, where $T_Y$ stands
  for the mixing time of the chain $Y$ estimated in \eqref{e:CPRW}. In
  order to estimate the right-hand side of \eqref{e:97}, we study the
  typical size of $R_{m_{\pm}\ell}$ where
  \begin{equation}
    \label{e:98}
    m_- = \left\lceil\ell^{-1}(u-\eta) \Cap_\Delta (B) \right\rceil
    \quad \text{and} \quad
    m_+ = \left\lfloor\ell^{-1}(u+\eta) \Cap_\Delta (B)\right\rfloor.
  \end{equation}
  From Lemma~\ref{l:baredelta} and \eqref{e:CPRW}, it follows that
  \begin{equation}
    \label{e:m_lower}
    m_\pm \geq c N^{d - 2 - \varepsilon}.
  \end{equation}

  Let $\mathcal G_i=\sigma (X_i:i\le R_{i\ell})$. Using the standard
  properties of the mixing time (see e.g.~\cite[Section~4.5]{LPW09}) and
  the strong Markov property, it is easy to see that
  \begin{equation}
    \|P[(X_{R_{i\ell}},X_{D_{i\ell}})\in \,\cdot\,|\mathcal G_{i-1}] - \pi
    (\cdot)\|_{TV} \le 2^{-N^{\varepsilon }}.
  \end{equation}
  By Lemma~\ref{l:baredelta},
  $\pi (\{y\}\times \partial \Delta ) = \bar e_B^\Delta (y)\asymp N^{1-d}$
  uniformly in $y\in\partial B$, and thus
  \begin{equation}
    \label{e:910}
    \bigg|\frac{P[X_{R_{i\ell}}=y|\mathcal G_{i-1}]}{\bar e_B^\Delta
      (y)}-1\bigg|
    \le c 2^{-N^{\varepsilon /2}}, \qquad i\ge 1.
  \end{equation}
  For $m$ standing for $m_+$ or $m_-$, we write
  \begin{equation}
    R_{m\ell} = \sum_{j=1}^{m} Z_j, \text{ where }
    Z_j = R_{j\ell}-R_{(j-1)\ell} \text{ and } R_0:=0.
  \end{equation}
  For every $j\ge 2$, by \eqref{e:910},
  \begin{equation}
    P[Z_j > t|\mathcal G_{j-2}] \leq (1+c 2^{-N^{\varepsilon /2}}) P_{\bar e_B^\Delta }[R_\ell > t]
    \leq 2 \ell P_{\bar e_B^\Delta }[R_1 > t/\ell].
  \end{equation}
  By the invariance principle $P[R_1>N^2]\le c <1$. Using this and
  Markov property iteratively yields
  $P[R_1>N^{2+\delta }]\le e^{-c N^\delta }$ for any $\delta >0$, and thus
  \begin{equation}
    \label{e:915}
    P[Z_j > \ell N^{2+\delta}|\mathcal G_{j-2}]
    \leq  2\ell P_{\bar e_B^\Delta }[R_1 > N^{2+\delta}]
    \leq c \exp\{-N^{c'  \delta}\}.
  \end{equation}
  Analogous reasoning proves also that
  \begin{equation}
    P[Z_1\ge \ell N^{2+\delta }]
    \leq c \exp\{-N^{c'  \delta}\}.
  \end{equation}
  Observe also that for $j\ge 2$, by \eqref{e:910} again,
  \begin{equation}
    |E[Z_j]-E[Z_j|\mathcal G_{j-1}]|\le c 2^{-N^{\varepsilon /2}}E(Z_j).
  \end{equation}
  Hence,
  \begin{equation}
    \begin{split}
      \label{e:917}
     & P[|R_{m\ell} - E(R_{m\ell})| > \eta E(R_{m\ell})]
      =P\Big[ \Big|\sum_{j=1}^{m}\big(Z_j - E[Z_j]\big)\Big| > \eta E(R_{m\ell}) \Big]
      \\ &
      \le P[Z_1\ge \eta  E(R_{m\ell})/4] +
      \sum_{n \in \{0,1\}} P\Big[ \Big|
        \sum_{\substack{1 \leq j \le m\\ j \text{ mod } 2 = n}}
        \big(Z_{j} - E[Z_{j}|\mathcal G_{j-2}]\big)\Big| >
        \eta E(R_{m\ell})/4 \Big].
    \end{split}
  \end{equation}
  Setting $\hat Z_j = Z_j \wedge \ell N^{2 + \delta}$, which by
  \eqref{e:915} satisfies
  \begin{equation}
    |E[\hat Z_j|\mathcal G_{j-2}] - E[Z_j|\mathcal G_{j-2}]|
    = \int_{\ell N^{2 + \delta}}^\infty P[Z_j > t|\mathcal G_{j-2}] dt
  \leq c \exp\{- N^{c'\delta }\},
  \end{equation}
  the right-hand side of \eqref{e:917} can be bounded by
  \begin{equation}
    \leq c m \exp\{-N^{c'  \delta}\} + \sum_{n\in{0,1}}P\Big[
    \Big|
    \sum_{\substack{1 \leq j \le m\\ j = n \text{ mod } 2}}
    \big(\hat{Z}_{j}
    - E[\hat Z_{j}|\mathcal G_{j-2}]\big)\Big| > \eta
    E(R_{m\ell})/4 \Big].
  \end{equation}
  Azuma's inequality together with $E[R_{m\ell}]\asymp N^d$,
  \eqref{e:98}, \eqref{e:m_lower}, and Lemma~\ref{l:baredelta} then yield
  \begin{equation}
    \begin{split}
      & \leq c m \exp\{-N^{c' \delta}\} +
      4 \exp\Big\{ - \frac{2 c(\eta E(R_{m\ell}))^2}{m (\ell N^{2 + \delta})^2}
        \Big\}\\
      & \leq c m \exp\{-N^{c'\delta}\} + 4 \exp\Big\{ - c\eta^2 \frac{m N^{2d - 4 -
            2\delta}}{\Cap_\Delta(B)^2} \Big\}\\
      & \leq c m \exp\{-N^{c' \delta}\}
      + 4 \exp\Big\{ - c \eta^2 N^{d-4 + 2\gamma - \varepsilon - 2\delta}
        \Big\}.
    \end{split}
  \end{equation}
  For every $d\ge 3$ and $\gamma $ as in \eqref{e:gammachi}, it is
  possible to fix $\delta $ and $\varepsilon $ sufficiently small so that
  the exponent of $N$ on the right-hand side of the last display is
  positive. Therefore the above decays at least as
  $C \exp\{-c \eta^2 N^c\}$ as $N$ tends to infinity, finishing the proof
  of the proposition.
\end{proof}

%\begin{equation}
%  \begin{array}{c}
%    \gamma = 3/4 \text{ and } \varepsilon, \delta < 1/10, \text{ if $d =
%      3$ and}\\
%    \text{and }\gamma > 0 \text{ and } \varepsilon, \delta < \gamma/3
%    \text{ if $d \geq 4$}.
%  \end{array}
%\end{equation}

We now count the number of excursions of random interlacements at level
$u$ into $B$. Let $J_u^N$ be the Poisson process with intensity
$\Cap (B_N)$ driving the excursions of random interlacements to $B_N$,
cf.~\eqref{e:ridef}. From Section~\ref{s:toruscoupling}, recall the
definition \eqref{e:Ti} of random variables $T^{(i)}$ giving the number
of excursions of $i$-th random walk between $B$ and $\Delta $. Given
those, denote by $\mathcal N'(u)$ the number of steps of Markov chain $Z$
corresponding to the level $u$ of random interlacements,
\begin{equation}
  \mathcal N'(u) = \sum_{i=1}^{J^N_u} T^{(i)}
\end{equation}

\begin{proposition}
  \label{p:nri}
  There exist constants $c$, $C$ depending only on $\gamma $ and $u$ such
  that for every $u>0$
  \begin{equation}
    P\big[|\mathcal N'(u) - u \Cap_\Delta (B)|\ge \eta u \Cap_\Delta (B)\big]\le
    C\exp\{-c \eta^2 N^c\}.
  \end{equation}
\end{proposition}

\begin{proof}
  By definition of random interlacements, $J^N_u$ is a Poisson random
  variable with parameter $u \Cap (B) \asymp u N^{d-2}$, and thus, by
  Chernov estimate,
  \begin{equation}
    \label{e:cria}
    P\big[|J^N_u - u \Cap (B)|\ge \eta u \Cap (B)\big]\le
    C\exp\{-c \eta^2 N^{d-2}\}.
  \end{equation}
  The random variables $T^{(i)}$ are i.i.d.~and stochastically dominated
  by the geometric distribution with parameter
  $\inf_{y\in \partial \Delta }\PZ_y[H_B=\infty] \asymp N^{\gamma -1}$,
  by Lemma~\ref{l:ebar}. Moreover, by summing
  \eqref{e:Tia}--\eqref{e:Tib} over $x_1\in\partial B$ we obtain
  \begin{equation}
    \EZ_{\bar e_B} T^{(i)} = \frac{\Cap_\Delta (B)}{\Cap (B)}.
  \end{equation}
  Applying Chernov bound again for $v=(1\pm\frac \eta 2)u\Cap (B)$,
  \begin{equation}
    \label{e:crib}
    P\Big[\Big|\sum_{i=1}^v T^{(i)} - \frac{v \Cap_\Delta (B)}{\Cap (B)}\Big|
    \ge \frac \eta 2 \, \frac{v \Cap_\Delta (B)}{\Cap (B)}
    \Big]\le C
    \exp\{-c \eta^2 N^c\}
  \end{equation}
  for some constants $C$ and $c$ depending on $\gamma $ and $u$.
  The proof is completed by combining \eqref{e:cria} and \eqref{e:crib}.
\end{proof}

\section{Proofs of the main results}%<<<1
\label{s:proofs}
We can now finally show our main results: Theorem~\ref{t:toruscoupling}
giving the coupling between the vacant sets of the random walk and the
random interlacements in macroscopic subsets of the torus, and
Theorem~\ref{t:phasetransition} implying the phase transition in the
behavior of the radius of the connected cluster of the vacant set of the
random walk containing the origin.

\begin{proof}[Proof of Theorem~\ref{t:toruscoupling}]
  As already announced several times, Theorem~\ref{t:couplingmc} is the
  key ingredient of this proof.

  Recall the definitions and transition probabilities of the Markov
  chains $Y=(Y_i)_{i\ge 1}$ and $Z=(Z_i)_{i\ge 1}$ from
  Section~\ref{s:coupling}. The state space $\Sigma $ of these Markov
  chains is finite, so we can apply Theorem~\ref{t:couplingmc} to
  construct a coupling of those two
  chains on some probability space
  $(\Omega_N , \mathcal F_N, \mathbb Q_N)$ carrying a Poisson point process with intensity $\mu \otimes dx$ on
  $\Sigma \times [0,\infty)$, so that their ranges coincide
  in sense of \eqref{e:goodmc}.
  We will apply this theorem with
  \begin{equation}
    \begin{aligned}
      n&= u \Cap_\Delta (B) \asymp N^{d-1-\gamma },
      &&\text{(cf.~Lemma~\ref{l:baredelta},
        Propositions~\ref{p:concentr_Nt}, \ref{p:nri})}
      \\|\Sigma|& \asymp N^{2(d-1)},
      \\ T_Y,T_Z&\le c N^{1-\gamma },
      &&\text{(Lemma~\ref{l:mixingtimes})}
      \\g(\boldsymbol z)&=\bar e_B^\Delta (z_1)\asymp N^{1-d},\qquad
      &&\text{(Lemma~\ref{l:baredelta})}
      \\\Var \rho_{\vec z}^Y,\Var \rho_{\vec z}^Z&\asymp N^{1-d}N^{-\gamma (d-1)},
      &&\text{(Lemma~\ref{l:varrho})}
      \\\|\rho^Y_{\vec z} \|_\infty, \|\rho^Z_{\vec z}\|_\infty&\asymp N^{-\gamma (d-1)},
      &&\text{(Lemma~\ref{l:hittingprobas}, cf.~\eqref{e:maxrho} and below)}
    \end{aligned}
  \end{equation}
  In addition, it follows from Claims~\ref{c:YregenA},~\ref{c:YregenB},
  that $\pi^\star$ decays polynomially with $N$, and thus
  \begin{equation}
    k(\varepsilon_N ) \sim c \log N - c'\log \varepsilon_N.
  \end{equation}
  Substituting those into condition \eqref{e:epsasmc} of
  Theorem~\ref{t:couplingmc} implies that
  $\varepsilon_N < c_0=c_0(d,\gamma ,\chi)$ as assumed in
  Theorem~\ref{t:toruscoupling}. If, in addition, $\varepsilon_N$
  satisfies $\varepsilon_N^2 \ge c N^{\delta-\kappa }$ for
  $\kappa = \gamma (d-1)-1>0$ and $\delta >0$, the, after some algebra,
  we obtain
  \begin{equation}
    \label{e:couplingZY}
    \mathbb Q_N\Big[\bigcup_{i\le (1-\varepsilon_N)n} Z_i
      \subset \bigcup_{i\le n} Y_i \subset
      \bigcup_{i\le (1+\varepsilon_N)n} Z_i\Big]\ge 1- c_1 e^{-c_2 N^{\delta '}}
  \end{equation}
  for some $\delta' $ as in Theorem~\ref{t:toruscoupling}.

  We now re-decorate $Y$ and $Z$ to obtain a coupling of the vacant sets
  restricted to $B$. Let $\Gamma $ be the space of all finite-length
  nearest-neighbor paths on $\mathbb T_N^d$. For $\gamma \in \Gamma $ we
  use $\ell(\gamma )$ to denote its length and write $\gamma $ as
  $(\gamma_0,\dots,\gamma_{\ell(\gamma )})$.

  To construct the vacant set of the random walk, we define on the same
  probability space $(\Omega_N , \mathcal F_N, \mathbb Q_N)$ (by possibly
    enlarging it) two sequences of `excursions' $(\mathcal E_i)_{i\ge 1}$
  and $(\tilde{\mathcal E}_i)_{i\ge 0}$, whose distribution is uniquely
  determined by the following properties
  \begin{itemize}
    \item Given realization of $Y=((Y_{i,1},Y_{i,2}))_{i\ge 1}$ and
    $Z=((Z_{i,1},Z_{i,2}))_{i\ge 1}$, $(\mathcal E_i)$ and
    $(\tilde{\mathcal E}_i)$ are conditionally independent sequences of
    conditionally independent random variables.

    \item For every $i\ge 1$, the random variable $\mathcal E_i$ is
    $\Gamma $-valued and for every $\gamma \in \Gamma$,
    \begin{equation}
      \mathbb Q_N[\mathcal E_i=\gamma|Y,Z]=
      P_{Y_{i,1}}[H_\Delta=\ell(\gamma ),X_i=\gamma_i\forall i\le \ell(\gamma )
      |X_{H_\Delta }=Y_{i,2}].
    \end{equation}

    \item For every $i\ge 1$, the random variable $\tilde{\mathcal E}_i$
    is $\Gamma $-valued and for every $\gamma \in \Gamma$,
    \begin{equation}
      \mathbb Q_N[\tilde{\mathcal E}_i=\gamma|Y,Z]=
      P_{Y_{i,2}}[H_B=\ell(\gamma ),X_i=\gamma_i\forall i\le \ell(\gamma )
      |X_{H_B}=Y_{i+1,1}].
    \end{equation}

    \item The random variable $\tilde{\mathcal E}_0$ is $\Gamma $-valued and
    \begin{equation}
      \mathbb Q_N[\tilde{\mathcal E}_0=\gamma|Y,Z]=
      P[R_1=\ell(\gamma ), X_i=\gamma_i\forall i\le \ell(\gamma )|X_{R_1}=Y_{1,1}].
    \end{equation}
  \end{itemize}
  With slight abuse of notation, we construct on
  $(\Omega_N , \mathcal F_N, \mathbb Q_N)$ a process $(X_n)_{n\ge 0}$
  defined by concatenation of
  $\tilde{\mathcal E_0}, \mathcal E_1, \tilde{\mathcal E}_1, \mathcal E_2, \dots$.
  From the construction it follows easily that $X$ is a simple random
  walk on $\mathbb T_N^d$ started from the uniform distribution. Finally,
  we write $R_1=\ell(\tilde{\mathcal E}_0)$,
  $D_1=\ell(\tilde{\mathcal E}_0)+\ell(\mathcal E_1)$, \dots, which is
  consistent with the previous notation, and set, as before,
  $\mathcal N(uN^d)=\sup\{i:R_i<uN^d\}$. Finally, we fix an arbitrary
  constant $\beta >0$ and define the vacant set of random walk on
  $(\Omega_N, \mathcal F_N, \mathbb Q_N)$ by
  \begin{equation}
    \label{e:cVNu}
    \mathcal V^u_N = \mathbb T_N^d \setminus \{X_{\beta
      N^d},\dots,X_{(\beta+u)N^d }\},
  \end{equation}
  which has the same distribution as the vacant set introduced in
  \eqref{e:VNu}, since $(X_i)$ is stationary Markov chain.

  To construct the vacant set of random interlacements intersected with
  $B$, let $\mathcal I_0=\emptyset$ and for $i\ge 1$ inductively
  \begin{equation}
    \begin{split}
      \label{e:matchingEs}
      \iota_i &= \inf\{j\ge 1: j\notin \mathcal I_{i-1}, Y_j=Z_i\},
      \\\mathcal E_i^{\RI}&= \mathcal E_{\iota_i},
      \\\mathcal I_i &= \mathcal I_{i-1}\cup \{\iota_i\}.
    \end{split}
  \end{equation}
  Let further $(U_i)_{i\ge 1}$ be a sequence of conditionally independent
  Bernoulli random variables with (cf.~\eqref{e:distZ})
  \begin{equation}
    P[U_i=1]=\frac{\PZ_{Z_{i,2}}[H_B=\infty]\bar e_B(Z_{i+1,1})}
    {\PZ_{Z_{i,2}}[H_B<\infty, X_{H_B}=Z_{i+1,1}] +
      \PZ_{Z_{i,2}}[H_B=\infty]\bar e_B(Z_{i+1,1})}.
  \end{equation}
  The event $\{U_i=1\}$ heuristically correspond to the event ``after the
  excursion $Z_i$ the random walk leaves to infinity and the excursion of
  random interlacements corresponding to $Z_{i+1}$ is a part of another
  random walk trajectory''. We set $V_0=0$ and inductively for $i\ge 1$.
  $V_i=\inf\{i>V_{i-1}:U_i=1\}$. Then, by construction, for every $i\ge 1$,
  $(\mathcal E_j^{\RI})_{V_{i-1}<j\le V_i}$ has the same distribution as
  the sequence of excursions of random walk $X^{(i)}$ into $B$,
  cf.~\eqref{e:ridef}, \eqref{e:riexc}. Finally, as in \eqref{e:ridef},
  we let $(J^N_u)_{u\ge 0}$ to stand for a Poisson process with intensity
  $\Cap (B)$, defined on $(\Omega_N, \mathcal F_N, \mathbb Q_N)$,
  independent of all previous randomness, and set
  \begin{equation}
    \mathcal N'(u)=V_{J_u^N}.
  \end{equation}
  This is again consistent with previous notation. Finally, for $\beta $
  as above, we can construct the random variables having the law of the
  vacant set of random interlacements at levels $u+\varepsilon_N$ and
  $u-\varepsilon_N$ intersected with $B$,
  \begin{equation}
    \label{e:cVu}
    \mathcal V^{u\pm\varepsilon_N} = B\setminus
    \bigcup_{i=\mathcal N'(\beta\mp\varepsilon N/2) }
    ^{\mathcal N'(\beta + u \pm \varepsilon_N/2)}
    \text{Range}(\mathcal E^{\RI}_i).
  \end{equation}

  Denoting $\mathcal K_N=\Cap_\Delta (B)$, by
  Proposition~\ref{p:concentr_Nt} the set $\mathcal V_N^u$ of
  \eqref{e:cVNu} satisfies
  \begin{equation}
    \label{e:erra}
    \mathbb Q_N\bigg[
      B_N \setminus
      \bigcup_{i=(\beta -\varepsilon_N/4)\mathcal K_N}
      ^{(\beta +u+\varepsilon_N/4)\mathcal K_N}
      \text{Range}(\mathcal E_i)
      \subset \mathcal V_N^u \subset
      B_N\setminus
      \bigcup_{i=(\beta +\varepsilon_N/4)\mathcal K_N}
      ^{(\beta +u-\varepsilon_N/4)\mathcal K_N}
      \text{Range}(\mathcal E_i)\bigg]
    \ge 1-Ce^{-c\varepsilon_N^2 N^{c}}.
  \end{equation}
  Combining \eqref{e:couplingZY} and \eqref{e:matchingEs} yields
  \begin{equation}
    \begin{split}
      &\mathbb Q_N\bigg[
        B_N \setminus
        \bigcup_{i=(\beta -\varepsilon_N/4)\mathcal K_N}
        ^{(\beta +u+\varepsilon_N/4)\mathcal K_N}
        \text{Range}(\mathcal E_i)
        \supset
        B_N \setminus
        \bigcup_{i=(\beta -\varepsilon_N/3)\mathcal K_N}
        ^{(\beta +u+\varepsilon_N/3)\mathcal K_N}
        \text{Range}(\mathcal E^{RI}_i)
        \bigg]
      \ge 1-Ce^{-c_2N^{\delta '}},
      \\&\mathbb Q_N\bigg[
        B_N \setminus
        \bigcup_{i=(\beta +\varepsilon_N/4)\mathcal K_N}
        ^{(\beta +u-\varepsilon_N/4)\mathcal K_N}
        \text{Range}(\mathcal E_i)
        \subset
        B_N \setminus
        \bigcup_{i=(\beta +\varepsilon_N/3)\mathcal K_N}
        ^{(\beta +u.\varepsilon_N/3)\mathcal K_N}
        \text{Range}(\mathcal E^{RI}_i)
        \bigg]
      \ge 1-Ce^{-c_2N^{\delta '}}.
    \end{split}
  \end{equation}
  Finally, by Proposition~\ref{p:nri}, for vacant sets as in \eqref{e:cVu},
  \begin{equation}
    \begin{split}
      \label{e:errc}
      &\mathbb Q_N\bigg[\mathcal V^{u+\varepsilon_N/2}\cap B \subset
        B_N \setminus
        \bigcup_{i=(\beta -\varepsilon_N/3)\mathcal K_N}
        ^{(\beta +u+\varepsilon_N/3)\mathcal K_N}
        \text{Range}(\mathcal E^{RI}_i)\bigg]
      \ge 1-Ce^{- c\varepsilon_N^2 N^c},
      \\&\mathbb Q_N\bigg[\mathcal V^{u-\varepsilon_N/2}\cap B \supset
        B_N \setminus
        \bigcup_{i=(\beta +\varepsilon_N/3)\mathcal K_N}
        ^{(\beta +u-\varepsilon_N/3)\mathcal K_N}
        \text{Range}(\mathcal E^{RI}_i)\bigg]
      \ge 1-Ce^{-c \varepsilon_N^2 N^c }.
    \end{split}
  \end{equation}
  Theorem~\ref{t:toruscoupling} then follows by combining
  \eqref{e:erra}--\eqref{e:errc}.
\end{proof}

\begin{proof}[Proof of Theorem~\ref{t:phasetransition}]
  Let us first introduce a simple notation. If $\mathcal{C}$ is a random
  subset of either $\mathbb{T}^d$ or $\mathbb{Z}^d$, let
  $A_N(\mathcal{C})$ stand for the event $[\diam(\mathcal{C}) > N/4]$,
  which appears in the definition of $\eta_N(u)$. We also denote by
  $\mathcal{C}_0(u)$ the connected component containing the origin of
  $\mathbb{Z}^d$ for random interlacements at level $u$.

  We now turn to the proof of \eqref{e:phasetransitionsubcritical}. Fix
  $u > u_\star(d)$.  Letting $u' \in (u_\star, u)$ and writing
  $u'=(1-\varepsilon )u$, we estimate
  \begin{equation}
    P[A_N(\mathcal{C}_N(u))] \leq 1 - \mathbb{Q}_N\Big[(\mathcal V_N^u \cap \mathcal B_N) \subset (\mathcal V^{u(1-\varepsilon )} \cap \mathcal B_N)\Big] + P\big[A_N(\mathcal{C}_0(u'))\big],
  \end{equation}
  which clearly tends to zero using Theorem~\ref{t:toruscouplingweak} and
  the fact that $u' > u_\star$.

  Now let us treat the supercritical case in
  \eqref{e:phasetransitionsupercritical}. Given $u < u_\star$ and
  $\varepsilon > 0$, we use the continuity of $\eta(u)$ in $[0,u_\star)$,
  see Corollary~1.2 of \cite{Tei09b}, to find $u'$ and $u''$ such that
  \begin{equation}
    (1-\varepsilon ) u \le u' < u < u'' \le (1+\varepsilon )u \quad \text{and} \quad \eta(u') - \eta(u'') < \varepsilon.
  \end{equation}
  We now observe that for $N > c$ we have
  $|\eta(u') - P[A_N(\mathcal C_0(u'))] | < \varepsilon$.
  Therefore, since $\eta $ is non-increasing function,
  \begin{equation}
    \begin{split}
      \big|P&[A_N(\mathcal{C}_N(u))]  - \eta(u)\big|
      \\&\leq \varepsilon +
      \big(P[A_N(\mathcal{C}_N(u))] - \eta(u'')\big)_-
      + \big(P[A_N(\mathcal{C}_N(u))] - \eta(u')\big)_+\\
      & \overset{{N > c}}\leq 2\varepsilon
      + \big( \mathbb{Q}[A_N(\mathcal C_N(u))] - \mathbb{Q} [A_N(\mathcal{C}_0(u''))]\big)_-
      + \big( \mathbb{Q}[A_N(\mathcal C_N(u))] - \mathbb{Q} [A_N(\mathcal{C}_0(u'))]\big)_+\\
      & \leq 2\varepsilon + 1 - \mathbb Q_N\big[(\mathcal V^{u(1+\varepsilon )}\cap \mathcal B_N) \subset (\mathcal V_N^u \cap \mathcal B_N) \subset (\mathcal V^{u(1-\varepsilon )}\cap \mathcal B_N)\big].
    \end{split}
  \end{equation}
  Since the limsup of the right-hand side of the above equation is at most
  $2 \varepsilon$ by Theorem~\ref{t:toruscouplingweak} and $\varepsilon > 0$ is
  arbitrary, we have proved \eqref{e:phasetransitionsupercritical} and
  consequently Theorem~\ref{t:phasetransition}.
\end{proof}

\appendix%<<<1
\section{A Chernov-type estimate for additive functionals of Markov chains}

We show here a simple variant of Chernov bound for additive functionals
of Markov chains. Many such bounds were obtained previously, but they do
not suite our purposes. E.g.,~Lezaud \cite{Lez98} (see also
  Theorems~2.1.8, 2.1.9 in \cite{Sal97}) provides such bounds in terms of
the spectral gap of the Markov chain. Since the spectral gap of
non-reversible Markov chains is not easy to estimate, and, more
importantly, it does not always reflect the mixing properties of the
chain, it seems preferable to use the mixing time of the chain as the
input. This idea was applied e.g.~in \cite{CLLM12}, whose bounds, in
contrary to \cite{Lez98}, do not use the information about the variance
of the additive functional under the equilibrium measure, and thus give
worse estimates in the case where this variance is known. The theorems
below can be viewed as combination of those two results.

We consider discrete time Markov chains first.
\begin{theorem}
  \label{t:appdisc}
  Let $(X_n)_{n\ge 0}$ be a discrete-time Markov chain on a finite state
  space $\Sigma $ with transition matrix $P$, initial distribution $\nu $,
  mixing time $T$, and invariant distribution~$\pi $. Then, for every
  $n\ge 1$, every function $f:\Sigma \to [-1,1]$ satisfying $\pi (f)=0$
  and $\pi (f^2)\le \sigma ^2$, and every
  $\gamma \le \sigma ^2\wedge \frac 12$
  \begin{equation}
    \label{e:chernovdiscrete}
    \mathbb P \Big[
    \sum_{i<n}  f(X_n)  \ge
    n\gamma \Big]
    \le 4
    \exp\Big\{-  \Big\lfloor\frac n {k(\gamma )T}-1\Big\rfloor
    \frac{\gamma^2}{ 6\sigma^2 } \Big\},
  \end{equation}
  with
  \begin{equation}
    k(\gamma )= -\log_2 ( \pi_\star \gamma^2/(6\sigma^2))
  \end{equation}
  and $\pi_\star = \min_{x\in \Sigma } \pi (x)$.
\end{theorem}

\begin{proof}
  Let $\tau = k(\gamma )T$. From \cite[Section 4.5]{LPW09} it follows
  that, for any initial distribution $\nu $,
  \begin{equation}
    \label{e:taustep}
    (1-\varepsilon ) \pi (x) \le  \mathbb P [X_\tau = x]  \le (1+\varepsilon )
    \pi (x),
  \end{equation}
  with $\varepsilon \le \gamma^2/(6\sigma^2)$. For $0\le k < \tau $,
  define $X^{(k)}_j= X_{k+\tau j}$, $j\ge 0$. For every $k$,
  $(X^{(k)}_j)_{j\ge 0}$ is a Markov chain with transition matrix
  $P^\tau $ and invariant distribution $\pi $. In view of
  \eqref{e:taustep}, $(X^{(k)}_j)_{j\ge 1}$ are close to being
  i.i.d.~with marginal $\pi $; the distribution of $X^{(k)}_0$ cannot be
  controlled in general.

  Writing $Y^{(k)}_n=\sum_{0\le i<(n-k)/\tau } f(X^{(k)}_i)$, with help
  of Jensen's inequality and the exponential Chebyshev bound, we have for
  every $\lambda >0$
  \begin{equation}
    \label{e:excessa}
    \mathbb P \Big[\sum_{j<n} f(X_j)\ge \gamma n\Big]\le
    \exp\big\{-\lambda \gamma n \tau^{-1} \big\}
    \frac 1 \tau
    \sum_{k<\tau } \mathbb E \big[\exp\{\lambda Y^{(k)}_n\}\big].
  \end{equation}
  Using \eqref{e:taustep}, the Markov property recursively, and the fact
  $f\le 1$ for the summand $f(X_0^{(k)})$,
  \begin{equation}
    \mathbb E\big[\exp\{\lambda  Y_n^{(k)}\}\big]\le e^\lambda
    \exp\Big\{
    \Big\lfloor \frac{n-k}{\tau }\Big\rfloor \big(\log (\pi (e^{\lambda
      f}))+\log
    (1+\varepsilon )\big)\Big\},
  \end{equation}
  for all $0\le k<\tau $. By Bennett's lemma (see e.g.
    \cite[Lemma~2.4.1]{DZ98}),
  \begin{equation}
    \pi(e^{\lambda f})\le \frac{1}{1+\sigma^2}\, e^{-\lambda \sigma^2} +
    \frac{\sigma^2}{1+\sigma^2}\, e^{\lambda }.
  \end{equation}
  Inserting this bound back into \eqref{e:excessa} and optimizing over
  $\lambda $ as in \cite[Corollary 2.4.7]{DZ98}, which amounts to choose
  \begin{equation}
    e^\lambda  =   \frac 1 {\sigma ^2}\cdot \frac{\gamma +\sigma ^2}{1-\gamma }
    \le 4,
  \end{equation}
  we obtain
  \begin{equation}
    \mathbb P \Big[\sum_{j<n} f(X_j)\ge \gamma n\Big]\le
    4 \exp\Big\{ - \Big\lfloor \frac{n}{\tau }-1\Big\rfloor
    \Big(
    H\Big(\frac{\gamma +\sigma ^2}{1+\sigma ^2}\Big|\frac{\sigma
      ^2}{1+\sigma ^2}\Big) - \log (1+\varepsilon)\Big)\Big\},
  \end{equation}
  where $H(x|p)=x\log\frac xp + (1-x)\log \frac{1-x}{1-p}$. Observing
  finally that for every $\sigma^2 \in (0,1)$ and $\gamma \in (0,\sigma^2)$
  \begin{equation}
    H\Big(\frac{\gamma +\sigma ^2}{1+\sigma ^2}\Big|\frac{\sigma
      ^2}{1+\sigma ^2}\Big) \ge  \frac
    {\gamma^2}{3\sigma ^2}
  \end{equation}
  and $\log (1+\varepsilon ) \le \varepsilon \le \gamma^2/(6\sigma^2)$,
  we obtain the claim of the theorem.
\end{proof}

For continuous-time Markov chains we have an analogous statement.

\begin{corollary}
  \label{t:appcont}
  Let $(X_t)_{t\ge 0}$ be a continuous-time Markov chain on a finite
  state space $\Sigma $ with generator $L$, initial distribution $\nu $,
  mixing time $T$, and invariant distribution $\pi $. Then for every $t>0$,
  every every function $f:\Sigma \to [-1,1]$ with $\pi (f)=0$ and
  $\pi (f^2)\le \sigma ^2$, and for $\gamma \le \sigma ^2\wedge \frac 12$
  \begin{equation}
    \mathbb P \Big[
      \int_0^t  f(X_s)\, \d s  \ge
      \gamma t  \Big]
    \le
     4
     \exp\Big\{-  \Big\lfloor\frac t {k(\gamma )T}-1\Big\rfloor
       \frac{\gamma^2}{6\sigma^2 } \Big\},
  \end{equation}
  with $k(\gamma )$ as in Theorem~\ref{t:appdisc}.
\end{corollary}
\begin{proof}
  The proof is a discretization argument: Consider a discrete-time Markov
  chain $Y^\delta_{n} = X_{\delta n}$. The mixing time $T(\delta )$ of
  $Y^\delta $ satisfies $T(\delta )=T \delta^{-1}(1+o(1))$ as
  $\delta \to 0$. The previous theorem applied with $n=\delta^{-1} t$,
  then implies
  \begin{equation}
    \mathbb P \Big[\delta \sum_{j< t \delta^{-1}} f(X_{j\delta })\ge \gamma
      t\Big]
    \le 4 \exp\Big\{-  \Big\lfloor\frac t {k(\gamma )T}-1\Big\rfloor
      \frac{\gamma^2}{ 6\sigma^2 } \Big\}.
  \end{equation}
  Taking $\delta \to 0$ and using the fact that $\Sigma $ is finite (that
    is the transition rates are bounded from below) yields the claim.
\end{proof}

Finally, let $h:\Sigma \to \mathbb R$ be an arbitrary function such that
$\Var_\pi (h)\le \sigma ^2$. Set
\begin{equation}
  f = (h - \pi(h))/ 2 \| h \|_\infty,
\end{equation}
so that $\|f\|_\infty \le 1$, $\pi(f)=0$ and
$\pi (f^2) \le \sigma ^2/(4 \|h\|_\infty^2)$. The corollary applied with
$\gamma = \delta \pi(h)/2 \| h \|_\infty$ then directly implies
\begin{equation}
  \label{e:apph}
  \mathbb P \Big[\int_0^t h(X_s)\, \d s - t \pi (h)
    \ge \delta t \pi (h)\Big]
  \le 4 \exp\Big\{-  \Big\lfloor\frac t {k'(\delta  )T}-1\Big\rfloor
    \frac{\delta^2 \pi (h)^2}{6 \sigma^2 } \Big\}
\end{equation}
with
\begin{equation}
  k'(\delta ) = -\log_2\big(\delta^2  \pi (h)^2 \pi_\star/(6 \sigma^2)\big)
\end{equation}
whenever
\begin{equation}
  \label{e:apphcond}
  \delta \le \frac{\sigma^2}{2 \pi (h) \|h\|_\infty} \wedge 1.
\end{equation}

%\bibliographystyle{jcamsalpha} %<<<1
%\bibliography{toruscoupling}

\end{document}